\documentclass[reqno,11pt,a4paper]{amsart}
\usepackage{amscd,amsmath,amssymb,verbatim,color,mathtools}
\usepackage{cite}
\usepackage{nicefrac}
\usepackage{color,xcolor,ucs}
\usepackage{hyperref}
\usepackage{url}
\usepackage{enumerate}
\usepackage{epsfig,subfigure}
\usepackage{appendix}
\usepackage{graphicx}
\usepackage{curve2e}
\usepackage{url}

\theoremstyle{plain}
\newtheorem*{MT}{Main Theorem}
\newtheorem{theorem}{Theorem}[section]

\newtheorem{lemma}[theorem]{Lemma}

\theoremstyle{remark}
\newtheorem{remark}{Remark}[section]
\newtheorem{example}{Example}[section]

\theoremstyle{definition}
\newtheorem{definition}{Definition}[section]

\usepackage{geometry}

\newcommand{\linspan}{\mathop{\rm span}\nolimits}

\newcommand{\diver}{\mathop{\rm div}\nolimits}
\newcommand{\curl}{\mathop{\rm curl}\nolimits}

\newcommand{\sign}{\mathop{\rm sign}\nolimits}

\newcommand{\rest}{\left.\kern-2\nulldelimiterspace\right|_}
\newcommand{\norm}[2]{|#1|_{#2}}
\newcommand{\fractx}[2]{{\textstyle\frac{#1}{#2}}}

\newcommand{\p}{\partial}

\newcommand{\e}{\varepsilon}
\newcommand{\ed}{\mathrm d}

\newcommand{\N}{{\mathbb N}}
\newcommand{\R}{{\mathbb R}}

\newcommand{\D}{{\mathrm D}}

\newcommand{\nnn}{\mathbf n}

\newcommand{\BB}{{\mathcal B}}
\newcommand{\CC}{{\mathcal C}}

\newcommand{\EE}{{\mathcal E}}
\newcommand{\FF}{{\mathcal F}}
\newcommand{\GG}{{\mathcal G}}

\newcommand{\KK}{{\mathcal K}}
\newcommand{\LL}{{\mathcal L}}

\newcommand{\RR}{{\mathcal R}}
\newcommand{\sS}{{\mathcal S}}

\newcommand{\YY}{{\mathcal Y}}
\newcommand{\ZZ}{{\mathcal Z}}

\newcommand*{\Bigcdot}{\raisebox{-.25ex}{\scalebox{1.25}{$\cdot$}}}

\begin{document}

\title{Approximate controllability for Navier--Stokes equations in~$\mathrm{3D}$ rectangles under Lions boundary conditions}
\author{Duy Phan}
\address{Duy Phan\newline \indent
Mathematics Laboratory, Tampere University of Technology,\newline\indent
P.O. Box 553, 33101 Tampere, Finland.\newline \indent
e-mail: {\small\tt duy.phan-duc@tut.fi}}
\author{S\'ergio S. Rodrigues}
\address{S\'ergio S. Rodrigues\newline \indent
Johann Radon Institute for Computational and Applied Mathematics, \"OAW,\newline\indent
Altenbergerstra{\normalfont\ss}e 69, 4040 Linz, Austria.\quad Tel. +43 732 2468 5241,\newline \indent
e-mail: {\small\tt sergio.rodrigues@ricam.oeaw.ac.at}}

\maketitle
\begin{abstract}
The $\mathrm{3D}$ Navier--Stokes system, under Lions boundary conditions, is proven to be approximately controllable provided a suitable saturating set does exist.
An explicit saturating set for $\mathrm{3D}$ rectangles is given.




\smallskip

\end{abstract}

\maketitle

\section{Introduction}
We consider the incompressible $\mathrm{3D}$ Navier--Stokes system in~$(0,\,T)\times\Omega$, under Lions boundary conditions,%
\begin{subequations}\label{sys-u-flat}
 \begin{align}
 \p_t u+\langle u\cdot\nabla\rangle u-\nu\Delta u+\nabla p+h&=\eta, &\diver u &=0,\label{sys-u-flat-eq}\\
 \left.\begin{pmatrix} u\cdot\nnn\\ \curl u-((\curl u)\cdot\nnn)\nnn\end{pmatrix}\right|_{\p\Omega}&=\begin{pmatrix}0\\0\end{pmatrix}, &u(0,\,x)&=u_0(x),
\end{align}
\end{subequations}
where~$\Omega\subset\R^3$ is a rectangle~$\Omega=(0,L_1)\times(0,L_2)\times(0,L_3)$, whose boundary is denoted by~$\p\Omega$. 
As usual $u=(u_1,u_2,u_3)$ and~$p$,
defined for $( t,x_1,x_2,x_3)\in I\times\Omega$, are respectively the unknown velocity field and
pressure of the fluid, $\nu>0$ is the viscosity, the operators~$\nabla$ and~$\Delta$ are respectively the
well known gradient and Laplacian in the space variables $(x_1,x_2,x_3)\in\Omega$,
$\langle u\cdot\nabla\rangle v$ stands for $(u\cdot\nabla v_1,u\cdot\nabla v_2,u\cdot\nabla v_3)$,
$\diver u\coloneqq \sum_{i=1}^3 \p_{x_i}u_i$, the vector~$\nnn$ stands for the outward unit normal vector to~$\p\Omega$, and $h$ is a fixed function. Finally, $\eta$ is a
control at our disposal.

Lions boundary conditions~(cf.~\cite[Section~6.9]{Lions69}) are a particular case of Navier boundary conditions.
For works and motivations concerning Lions and Navier boundary conditions (in both $\mathrm{2D}$ and $\mathrm{3D}$ cases) we
refer to~\cite{XiaoXin07,XiaoXin13,PhanRod17,ChemetovCiprianoGavrilyuk10,Kelliher06,IlyinTiti06} and references therein.

\subsection{The evolutionary system}\label{sS:evol}
We can rewrite system~\eqref{sys-u-flat} as an evolutionary system
\begin{equation}\label{sys_u}
 \dot u+A u+B(u,u)+h=\eta,\quad u(0)=u_0,
\end{equation}
in the subspace~$H\coloneqq\{u\in L^2(\Omega,\,\R^3)\mid\diver u=0\mbox{ and }(u\cdot\nnn)\rest{\p\Omega}=0\}$ of divergence free vector fields
which are tangent to the boundary. We may suppose that~$h$ and~$\eta$ take their values in~$H$ (otherwise we just take their orthogonal projections onto~$H$). We consider $H$, endowed with the norm inherited from $L^2(\Omega,\,\R^3)$, as a pivot space,
that is, $H=H'$. Further we set the spaces
\begin{align*}
 V&\coloneqq\{u\in H^1(\Omega,\,\R^3\mid u\in H\},\\
 \D(A)&\coloneqq\{u\in H^2(\Omega,\,\R^3)\mid u\in H, \quad\curl u-((\curl u)\cdot\nnn)\nnn\rest{\p\Omega}=0\}
\end{align*}

Above, for $u,v,w\in V$,
\begin{align}
 A&\colon V\to V',&\langle A u,v\rangle_{V',V}&\coloneqq \nu(\curl u,\curl v)_{L^2(\Omega,\,\R^3)},\label{LStokes}\\
 B&\colon V\times V\to V',&\langle B(u,v),w\rangle_{V',V}&\coloneqq -\int_\Omega(\langle u\cdot\nabla\rangle w)\cdot v\,\ed\Omega.\label{Bop}
\end{align}

It turns out that~$\D(A)=\{u\in H\mid Au\in H\}$ is the domain of~$A$. We will refer to~$A$ as the Stokes operator, under Lions boundary conditions. Further,
we have the continuous, dense, and compact inclusions~$\D(A)\xhookrightarrow{\rm d,c}V\xhookrightarrow{\rm d,c}H$. 

Denoting by~$\Pi$ the orthogonal projection in $L^2(\Omega,\,\R^3)$ onto $H$, for~$u,v\in\D(A)$ 
we may write~$Au\coloneqq\Pi(\nu\Delta u)$,
and $B(u,\,v)\coloneqq\Pi(\langle u\cdot\nabla\rangle v)$.

Further~$A$ maps~$V$ onto~$V'$, and the operator $A^{-1}\in\LL(H)$ is compact. The eigenvalues of $A$,
repeated accordingly with their multiplicity, form an increasing sequence $(\underline\lambda_k)_{k\in\N_0}$,
\[
 0<\underline\lambda_1\le\underline\lambda_2\le\underline\lambda_3\le\underline\lambda_4\le\dots,
\]
with $\underline\lambda_k$ going to $+\infty$ with $k$. 

\begin{remark}
 It is clear that the Stokes operator~\eqref{LStokes} is well defined, mapping~$V$ into~$V'$.
 We also see that the bilinear operator~\eqref{Bop} maps~$V\times V$ into~$V'$, due to the estimate

 \begin{align*}
  \langle B(u,v),w\rangle_{V',V} &\le {\rm C}_1\norm{u}{L^6(\Omega,\R^3)}\norm{\nabla w}{L^2(\Omega,\R^9)}\norm{v}{L^3(\Omega,\R^3)} \\
  &\le {\rm C}_2\norm{u}{H^1(\Omega,\R^3)}
 \norm{w}{H^1(\Omega,\R^3)}\norm{v}{H^1(\Omega,\R^3)}
 \end{align*}
 For further estimations on the bilinear operator we refer to~\cite[Section~2.3]{Temam95}.
\end{remark}

\subsection{Saturating sets and approximate controllability}\label{sS:saturH}

In the pioneering work~\cite{AgraSary05} the authors introduced a method which led to the controllability of finite-dimensional Galerkin approximations of
the~$\mathrm{2D}$ and~$\mathrm{3D}$ Navier--Stokes system, and to the approximate controllability of the~$\mathrm{2D}$ Navier--Stokes system, by means of
low modes/degenerate forcing.

Hereafter~$U\subseteq H$ will stand for a linear subspace of~$H$, and we denote
\[
 \BB(a,b)\coloneqq B(a,b)+B(b,a).
\]
\begin{definition}\label{D:FF-l}
Let~$\CC=\{W_k\mid k\in\{1,\,2,\,\dots,\,M\}\}$ and let~$E$ be a finite-dimensional space so that~$\CC\subset E\subset U$. The finite-dimensional subspace $\FF_{\tt L}(E)\subset U$ is given by
\[
\FF_{\tt L}(E)\coloneqq E+\linspan\{\BB(a,b)\mid a\in\CC,\,b\in E,\,\mbox{ and } (B(a,a),B(b,b))\in H\times H\}\textstyle\bigcap  U,
\]
\end{definition}

\begin{definition}\label{D:satur-l}
A given finite subset~$\CC=\{W_k\mid k\in\{1,\,2,\,\dots,\,M\}\}
 \subset U$
 is said $({\tt L},U)$-saturating if for the following sequence of subspaces
 ~$\GG^j\subset U$, defined recursively by
 \begin{align*}
  \GG^0&\coloneqq\linspan\CC,\qquad \GG^{j+1}\coloneqq\FF_{\tt L}(\GG^{j}),
\end{align*}
we have that the union~$\mathop{\bigcup}\limits_{j\in\N}\GG^j$ is dense in~$H$.
\end{definition}

In~\cite[Section~4]{AgraSary06} an explicit saturating set with~$4$ elements
is presented for the~$\mathrm{2D}$ Navier--Stokes system under periodic boundary conditions.

\begin{remark}
In order to deal with different types of boundary conditions and domains the definitions of saturating set has been slightly changed/relaxed in several works.
The definition of saturating set in~\cite[Section~4]{AgraSary06} is slightly different from Definition~\ref{D:satur-l}.
But, we can prove (cf.~\cite[Section~6.1]{Rod-Thesis08})
that the saturating set presented 
in~\cite{AgraSary06} is also $({\tt L},\D(A))$-saturating (cf.~\cite[Definition~2.2.1]{Rod-Thesis08}). Actually, in~\cite{AgraSary06} saturating sets are defined through the
frequencies (say, indexes) of eigenfunctions of the Stokes operator, but we can rewrite the definition in terms of the eigenfunctions themselves.
\end{remark}

We would like to refer also to
the works~\cite{Romito04,HairerMatt06,EMat01}, where the notion of saturating set was used to
derive ergodicity for the Navier--Stokes system under degenerate stochastic forcing (compare the sequence of subsets $\ZZ_n$ in~\cite[Section~4]{HairerMatt06} with the
sequence of subsets $\KK^n$ in~\cite[Section~8]{AgraSary05}).

In the pioneering work~\cite{AgraSary05} the set~$U$ in~\eqref{D:satur-l} is taken to be~$\D(A)$, the same is done
in~\cite{AgraSary06,Rod-Sev05,Rod06,Shirikyan06}. Later, in~\cite{Rod-Thesis08,Rod-wmctf07,PhanRod-ecc15}, $U$ is taken as~$V$ in order to delal either with Navier-type boundary
conditions or with internal controls supported in a small subset.

Often, for~$\mathrm{2D}$ Navier--Stokes equations and~$\mathrm{1D}$ Burgers equations, we have estimates for the bilinear term~$B(\Bigcdot,\Bigcdot)$ which allow us to
derive the well-posedness of the Cauchy problem, or that we can use to derive the controllability results. For example, the estimate 
\begin{align*}
\norm{\langle B(z+y,\,y),\,z\rangle_{V',\,V}}{\R}\le {\rm C}_1\norm{z}{H}\norm{z}{V}\norm{y}{V}+{\rm C}_1\norm{y}{H}\norm{y}{V}\norm{z}{V}
\end{align*}
is used in~\cite{PhanRod-ecc15} in the $\mathrm{1D}$ and~$\mathrm{2D}$ settings, to derive approximate controllability results for the Navier--Stokes and Burgers equations.
The same estimate~$$\norm{\langle B(z,\,y),\,z\rangle_{V',\,V}}{\R} \le {\rm C}_1\norm{z}{H}\norm{z}{V}\norm{y}{V}$$ above can be used to prove the
uniqueness of weak solutions for the corresponding systems. The estimate does not hold in the $\mathrm{3D}$ case.

In~\cite{Shirikyan06}, the method introduced in~\cite{AgraSary05} is developed so the case where the
well-posedness of the Cauchy problem is not known.

\begin{definition}\label{D:FF}
Given a finite dimensional space $E\subset U$. The finite-dimensional \\$\FF_{\tt B}(E)$ is the largest linear subspace~$F\subset U$ so that any~$\eta_1\in F$ can be written as
\[\eta_1=\eta-\sum_{j=1}^k\alpha_j B(\zeta^j),\]
with $k\in\N_0$, $(\eta,\zeta^1,\dots,\zeta^k)\in E^{1+k}$, and $(\alpha_1,\dots,\alpha_k)\in [0,+\infty)^k$.
\end{definition}

\begin{definition}\label{D:satur-b}
 A given finite subset~$\CC=\{W_k\mid k\in\{1,\,2,\,\dots,\,M\}\}
 \subset U$
 is said $({\tt B},U)$-saturating if for the following sequence of subspaces
 of~$\EE^{j}\subset U$, defined recursively by
 \begin{align*}
  \EE^0&\coloneqq\linspan\CC,\qquad \EE^{j+1}\coloneqq\FF_{\tt B}(\EE^{j}),
 \end{align*}
we have that the union~$\mathop{\bigcup}\limits_{j\in\N}\EE^j$ is dense in~$H$.
\end{definition}

Though, in~\cite{Shirikyan06} the author focuses on no-slip boundary conditions, $u\rest{\p\Omega}=0$, the results also hold
for other boundary conditions. This is also mentioned in~\cite[Section 2.3. Remark~2.7]{Shirikyan06} where
the author considers the case of periodic boundary conditions, and presents an
explicit~$({\tt B},\D(A))$-saturating set~$\CC$ (for the case of~$(1,1,1)$-periodic vectors) whose~$64$ elements are
eigenfunctions of the Stokes operator (i.e., the Laplacian).
For a general period~$q=(q_1,q_2,q_3)\in(\R_0)^3$ the existence of a saturating set is also proven~\cite[Section 2.3, Theorem~2.5]{Shirikyan06},
though the form of the saturating set is less explicit. 

Following the proof of the main Theorem~2.2 in~\cite{Shirikyan06} we can see that the result holds for a generic setting where we have the subspaces
\begin{gather*}
 \D(A)\xhookrightarrow{\rm d,c} V=\D(A^\frac{1}{2})\xhookrightarrow{\rm d,c} H=H',\quad V\subset H\cap H^1(\Omega,\R^3),\quad \D(A)\subset H\cap H^2(\Omega,\R^3),
\end{gather*}
with~$\D(A)=\{u\in H\mid Au\in H\}$ being the domain Stokes operator~$A$ (which depends on the boundary conditions), and where
the scalar products
\[
 \langle Au,v\rangle_{V',V}\quad\mbox{and}\quad (Au,Av)_H
\]
induce norms in~$V$ and~$\D(A)$, respectively, which are equivalent to the those inherited from~$H^1(\Omega,\R^3)$ and~$H^2(\Omega,\R^3)$, respectively.
\begin{remark}
The notation~$S\xhookrightarrow{}R$ above means that the inclusion~$S\subseteq R$ is continuous. The letter ``${\rm d}$'' (resp.~``${\rm c}$'') means that, in addition, the inclusion
is also dense (resp.~compact).
\end{remark}
\begin{remark}
 In the periodic case mentioned above, usually we take a smaller subspace~$H_{\rm per}\subset H$ in order to factor out
 the kernel of~$A$ (as an operator in~$H$), and guarantee
 that $(u,v)\mapsto\langle Au,v\rangle_{V_{\rm per}',V_{\rm per}}$ defines a scalar product in~$V_{\rm per}\coloneqq V\bigcap H_{\rm per}$.
 Notice that, for a nonzero constant vector field~$u$, and under periodic boundary conditions, we will have
 ~$Au=-\nu\Delta u=0$ and thus~$\langle Au,u\rangle_{V',V}=0$. Hence, $\langle Au,v\rangle_{V',V}$ does not define a scalar product in~$V=H\cap H^1(\Omega,\R^3)$.
\end{remark}

In particular, the results in~\cite{Shirikyan06} hold true for Lions boundary conditions, and we can conclude that approximate controllability for~$\mathrm{3D}$ Navier--Stokes
equation follows from the existence of
a~$({\tt B},\D(A))$-saturating set.

In this paper, we prove that approximate controllability also follows from the existence of
a~$({\tt L},\D(A))$-saturating set. Namely, we will prove the following.
\begin{MT}
Let $(u_0,\hat u)\in V\times V$, ~$\varepsilon>0$, and~$T > 0$. If~$\CC$ is a $({\tt L},\D(A))$-saturating set,
then we can find a control~$\eta\in L^\infty((0,T),\GG^1)$ so that the solution of system~\eqref{sys_u} satisfies
$\norm{u(T)-\hat u}{V}<\varepsilon$.
\end{MT}

Further, for any given length triplet~$L=(L_1,L_2,L_3)$, we present an explicit $({\tt L},\D(A))$-saturating set~$\CC$ for the~$\mathrm{3D}$
rectangle~$\Omega=(0,L_1)\times(0,L_2)\times(0,L_3)$. The elements of~$\CC$ are~$81$ eigenfunctions of the Stokes operator,
under Lions boundary conditions (cf. Theorem~\ref{T:satur3Drect} hereafter). 
Though it is not our goal here to find a saturating set with the minimum number of elements
as possible, we must say that for some~$L$ (maybe, even for all~$L$)
it may exist a saturating set with less elements. In any case, we underline that the existence of
a $({\tt L},\D(A))$-saturating set~$\CC$ is independent of the viscosity coefficient~$\nu$. In particular, the
linear space~$\GG^1$, where the control~$\eta$ takes its values in, does not change with~$\nu$.

Finally, we recall that in~\cite{Rod-wmctf07,Rod06} an explicit saturating set was found for a~$\mathrm{2D}$ rectangle $\Omega=(0,L_1)\times(0,L_2)$ with~$8$ elements.
In~\cite{PhanRod-ecc15} a saturating set with~$24$ elements is presented for the~$\mathrm{2D}$ Navier--Stokes system in a Cylinder under Lions boundary conditions
i.e., in a channel with Lions boundary conditions in the bounded direction and with periodicity assumption in the unbounded direction).

\begin{remark}
The ``${\tt L}$'' subcript in Definition~\ref{D:satur-l} underlines the fact that the {\em linearization}~$\BB$ of~$B$ is used in the recursion step, while in
The ``${\tt B}$'' subcript in Definition~\ref{D:satur-b} underlines the fact that the {\em bilinear} operator~$B$ is used in the recursion step.
\end{remark}

\subsection{Motivation and further references}
An advantage for considering~$({\tt L},\D(A))$-saturating sets is that the construction of~$\FF_{\tt L}(E)$ is easier than
the construction of~$\FF_{\tt B}(E)$. This is important, when we
need to dwell with explicit computations as in the case when we look for explicit saturating sets. Often, the existence of saturating sets is proven by showing that a
given explicit set is saturating, which involve essentially explicit computations
(Theorem~2.5 in~\cite{Shirikyan06} is an exception, but the proof is still strongly based on explicit computations).

For further results concerning the controllability and approximate controllability of Navier--Stokes (and also other) systems by a control with low finite-dimensional range
(independent of the viscosity coefficient) in several domains (including the $\mathrm{2D}$ Sphere and Hemisphere) we refer the
reader to~\cite{Nersisyan10,Nersisyan11,Nersesyan10,Shirikyan08,Shirikyan_Evian07,Shirikyan07,Sarychev12,AgraSary06,AgraSary08,AgrKukSarShir07}.
We also mention Problem~VII raised by A.~Agrachev in~\cite{Agrachev13_arx} where the author inquires about the achievable controllability properties for 
controls taking values in a saturating set whose elements are localized/supported in a small subset~$\omega\subset\Omega$. The existence of such saturating sets is an open
question (except for~$\mathrm{1D}$ Burgers in ~\cite{PhanRod-ecc15}). The controllability properties implied by such saturating set is an open question. There are some negative
results, as for example in the case we consider the~$\mathrm{1D}$ Burgers
equations in~$\Omega=(0,1)$ and take controls in~$L^2(\omega,\R)$, $w\subset\Omega$, the approximate controllability fails to hold. Instead, to drive the system
from one state~$u_0=u(0)$ at time~$t=0$ to another one~$u_T=u(T)$ at time~$t=T$, we may need $T$ to be big enough. Though we do not consider localized controls here, we refer
the reader to the related results in~\cite{FerGue07,Diaz96,Shirikyan17} and references therein.

Finally we would like to mention that in previous works the existence of a saturating set implied the exact controllability of Galerkin approximations and also the
exact controllability onto finite dimensional projections, see for example~\cite{AgraSary05}. To prove these results some geometric control tools are used. We refer
also to~\cite{ChambMasoSigaBosc09} where the approximate controllability is derived from controllability of Galerkin approximations.

\bigskip
The rest of the paper is organized as follows. In Section~\ref{S:AppControl} we prove that the existence of a~$({\tt L},\D(A))$-saturating set implies the approximate controllability of the
Navier--Stokes system. In Section~\ref{S:Rectangle} we present a $({\tt L},\D(A))$-saturating set.

\section{Approximate controllability}\label{S:AppControl}
As we said above, in~\cite{Shirikyan06} it is proven that the existence of a $({\tt B},\D(A))$-saturating set implies the
approximate  controllability of the $\mathrm{3D}$~Navier--Stokes system, at time~$T>0$. Here we prove that
we can conclude the same controllability property from the existence of a $({\tt L},\D(A))$-saturating set.

We recall now some definitions from~\cite{Shirikyan06}. Hereafter $u_0\in V$, ~$h\in L^2_{\rm loc}(\R_0,H)$, and~$E\subset\D(A)$ is a finite-dimensional subspace.
Let us consider the system
\begin{equation}\label{sys_u_E}
 \dot u+A u+B(u,u)+h=\eta,\quad u(0)=u_0, 
\end{equation}
where the control~$\eta$ takes its values in~$E$.

For simplicity we will denote
\[
 I_T\coloneqq(0,T),\quad\mbox{and}\quad \overline{I_T}\coloneqq [0,T],\qquad T>0.
\]

\begin{definition}
 Let $T$ be a positive constant. System~\eqref{sys_u_E} is said to be $E$-approximately controllable in time~$T$
 if for any~$\e > 0$ and any pair~$(u_0,\hat u)\in V\times \D(A)$, there exists a control function
$\eta\in L^\infty(I_T, E)$ and a corresponding solution $u\in C(\overline{I_T},V)\bigcap L^2(I_T, \D(A))$, such that
$\norm{u(T)-\hat u}{V} <\e$.
\end{definition}

\begin{definition}
Let $T$, $R$, and~$\e$ be positive constants. System~\eqref{sys_u_E}
is said to be $(\e, R,E)$-controllable in time~$T$ if for any $(u_0,\hat u)\in V\times \D(A)$ satisfying $\norm{u_0}{V}\le R\ge\norm{\hat u}{\D(A)}$,  there exists
a control function $\eta\in L^\infty(I_T, E)$ and a corresponding solution $u\in C(\overline{I_T},V)\bigcap L^2(I_T, \D(A))$ such that~$\norm{u(T)-\hat u}{V} <\e$.
\end{definition}

Recall the sequence in Definition~\ref{D:satur-b}. In~\cite[Section~2]{Shirikyan06} we find the following results.
\begin{theorem}\label{T:AppCont-b}
If~$\CC$ is a $({\tt B},\D(A))$-saturating set,
then for any positive~$T > 0$ the system~\eqref{sys_u_E} is $\EE^0$-approximately controllable at time~$T$.
\end{theorem}
\begin{theorem}\label{T:IndE-E1-b}
Let $T$, $R$, and~$\e$ be positive constants. Then system~\eqref{sys_u_E} is
$(\e, R, E)$-controllable in time~$T$ if it is $(\e, R, \FF_{\tt B}(E))$-controllable at time~$T$.
\end{theorem}

Recall the sequence in Definition~\ref{D:satur-l}. Here we prove the following.
\begin{theorem}\label{T:AppCont-l}
If~$\CC$ is a $({\tt L},\D(A))$-saturating set,
then for any positive~$T > 0$ the system~\eqref{sys_u_E} is $\GG^1$-approximately controllable at time~$T$.
\end{theorem}
\begin{proof}
Proceeding as in~\cite[Section~2.2]{Shirikyan06} we can prove that system~\eqref{sys_u_E} is $\GG^{j+1}$-approximately controllable in time~$T$, provided~$j\in\N$ is big enough.
Now for any~$u=\sum\limits_{i=1}^M u_i W_i\in\GG^0$, we have that~$B(u,u)=\sum\limits_{i=1}^M u_iB(W_i,u)$, which implies that
$B(u,u)\in\FF_{\tt L}(\GG^0)=\GG^1$, for all~$j\ge0$. Since~$\GG^j\subseteq\GG^{j+1}$, for all~$j\ge0$, we have that~$\GG^1+\FF_{\tt L}(\GG^{j})=\GG^1+\GG^{j+1}=\GG^{j+1}$.
By the following Theorem~\ref{T:IndE-E1-l} it follows that system~\eqref{sys_u_E} is
$(\GG^1+\GG^{j})$-approximately controllable in time~$T$. Repeating the last argument, we conclude that system~\eqref{sys_u_E} is
$(\GG^1+\GG^{1})$-approximately controllable in time~$T$.
\null\hfill\null{\qed}\end{proof}
\begin{theorem}\label{T:IndE-E1-l}
Let $T$, $R$, and~$\e$ be positive constants. Then system~\eqref{sys_u_E} is
$(\e, R, \GG^1+E)$-controllable in time~$T$ if it is $(\e, R, \GG^1+\FF_{\tt L}(E))$-controllable in time~$T$.
\end{theorem}
\begin{proof}
 Let us fix~$\hat\e>0$ and $\hat u\in \D(A)$. Let also~$(\xi_0,\xi_1)\in L^\infty(I_T,\GG^1)\times L^\infty(I_T,\FF_{\tt L}(E))$ be such that the corresponding solution for
 \begin{equation}\label{sys_u_01}
 \dot u+A u+B(u,u)+h=\xi_0+\xi_1,\quad u(0)=u_0, 
\end{equation}
satisfies
\begin{equation}\label{uThatu}
 \norm{u(T)-\hat u}{V}\le \hat{\e}. 
\end{equation}
We may write, for any $\rho>0$,
\begin{align*}
 \xi_1&=\eta+\sum_{i=1}^k\BB(a_i,b_i)=\eta+\sum_{i=1}^k \left(-B(\rho a_i-\rho^{-1}b_i)+\rho^2B(a_i)+\rho^{-2}B(b_i)\right)\\
 \end{align*}
 for suitable $k\in\N_0$, $\eta\in L^\infty(I_T,E)$, and suitable pairs $(a_i,b_i)\in L^\infty(I_T,\CC\times E)$. Therefore,
 \begin{align*}
 \xi_1&=\rho^2\eta_a +\eta_\rho + \rho^{-2}\eta_b
 \end{align*}
 with
 \begin{subequations}\label{etas0rho2}
 \begin{align}
 \eta_a&\coloneqq\sum_{i=1}^k B(a_i,a_i),\qquad \eta_b\coloneqq \sum_{i=1}^k B(b_i,b_i).\\
 \eta_\rho&\coloneqq \eta-\sum_{i=1}^k B(\rho a_i-\rho^{-1}b_i,\rho a_i-\rho^{-1}b_i).
\end{align} 
 \end{subequations}

Now we rewrite~\eqref{sys_u_01} as
\begin{equation}\label{sys_u_01a}
 \dot u+A u+B(u,u)+h=\xi_0+\rho^2\eta_a +\eta_\rho +\rho^{-2}\eta_b,\quad u(0)=u_0. 
\end{equation}
Since~\eqref{sys_u_01a} coincides with~\eqref{sys_u_01}, the solution~$u$ of~\eqref{sys_u_01a}
is independent of~$\rho$. Let us now consider the solution of the system
\begin{equation*}
 \dot w_\rho+A w_\rho+B(w_\rho,w_\rho)+h=\xi_0+\rho^2\eta_a +\eta_\rho ,\quad w_\rho(0)=u_0, 
\end{equation*}

The solution~$u$ is known (by Theorem's assumption) to exist for~$t\in\overline{I_T}$. We show now that the solution~$w_\rho$ also exists
for time~$t\in\overline{I_T}$, provided~$\rho$ is big enough.

Indeed, the difference~$z= u-w_\rho$ solves
\begin{equation}\label{sys_z}
 \dot z+Az+B(z,z)+\BB(u,z)=\rho^{-2}\eta_b,\quad z(0)=0, 
\end{equation}
and we know that~$u\in C(\overline{I_T},V)\subset L^4(I_T,H^1(\Omega,\R^3))$ and $\rho^{-2}\eta_b\in L^\infty(I_T,H)\subset L^2(I_T,H)$. Further we know that~$\hat z=0$ solves
system~\eqref{sys_z} with~$\eta_b=0$, for time~$t\in I_T$.
Therefore, from~\cite[Remark~1.9]{Shirikyan06}, we can conclude that there exists a unique solution for system~\eqref{sys_z}, for time~$t\in I_T$, provided
$\norm{\rho^{-2}\eta_b-0}{L^2(I_T,H)}$ is small enough. That is, provided~$\rho$ is big enough. Furthermore, we have that
\[
 \norm{u-w_\rho}{C(\overline{I_T},V)\bigcap L^2({I_T},\D(A))}=\norm{z}{C(\overline{I_T},V)\bigcap L^2({I_T},\D(A))}\le C \rho^{-2}\norm{\eta_b}{L^2(I_T,H)}
\]
for a suitable constant~$C$ depending only on~$\norm{u}{C(\overline{I_T},V)\bigcap L^2({I_T},\D(A))}$. See again~\cite[Remark~1.9]{Shirikyan06}.
In particular, for big enough $\rho>0$, we will have
\begin{equation}\label{wTuT}
 \norm{w_\rho(T)-u(T)}{V}\le \hat{\e}. 
\end{equation}

Observe that $\eta_\rho$ in~\eqref{etas0rho2} is in~$E-{\rm conv}\{B(e,e)\mid e\in E\}$,
where ${\rm conv}S$ stands for the convex cone generated by the subset~$S$, that is,
\[
 {\rm conv}S\coloneqq\left\{\sum_{i=1}^k\alpha_is_i\mid k\in\N,\, \alpha_i>0,\,s_i\in S\right\}.
\]
Hence by Proposition~3.2 in~\cite{Shirikyan06} there is $(\tilde\eta,\tilde\zeta)\in (L^\infty(I_T,E))^2$ so that
the corresponding solution for
\begin{equation*}
 \dot y_\rho+A(y_\rho+\tilde\zeta)+B(y_\rho+\tilde\zeta,y_\rho+\tilde\zeta)+h=\xi_0+\rho^2\eta_a +\tilde\eta,\quad y_\rho(0)=u_0, 
\end{equation*}
satisfies
\begin{equation}\label{wTyT}
 \norm{w_\rho-y_\rho}{C(\overline{I_T},V)}\le \hat{\e}. 
\end{equation}

\begin{remark} Actually, in~\cite[Proposition~3.2]{Shirikyan06}, it is assumed that~$\eta_\rho\in\FF_{\tt B}(E)$,
but following the proof in~\cite[Section~3.3]{Shirikyan06}, we can see that the 
 the proof is brought to the ``imitation'' (in short time intervals) of a constant control~$\eta_\rho\in E-{\rm conv}\{B(e,e)\mid e\in E\}$ (see also~\cite[Section~4.2, proof of Lemma~3.3]{Shirikyan06}).
\end{remark}

Now from~\cite[Proposition~3.1]{Shirikyan06} it follows that there exists a control~$\widehat\eta\in L^\infty(I_T,E)$ such that the solution of the system
\begin{equation*}
 \dot{\widehat y}_\rho+A\widehat y_\rho+B(\widehat y_\rho,\widehat y_\rho)+h=\xi_0+\rho^2\eta_a+\widehat\eta,\quad \widehat w_\rho(0)=u_0, 
\end{equation*}
satisfies
\begin{equation}\label{yThatyT}
 \norm{\widehat y_\rho(T)-y_\rho(T)}{V}\le \hat{\e}. 
\end{equation}

Finally, we observe that $\xi_0+\rho^2\eta_a+\widehat\eta\in L^\infty(I_T,\GG^1+E)$, and
\[
\norm{\widehat y_\rho(T)-\hat u}{V}\le 4\hat{\e},
\]
which can be concluded from~\eqref{uThatu}, \eqref{wTuT}, \eqref{wTyT}, and~\eqref{yThatyT}.
\null\hfill\null{\qed}\end{proof}

\section{The saturating set}\label{S:Rectangle}
Here we present a $({\tt L},\D(A))$-saturating set which consists of a finite number suitable eigenfunctions of the Stokes operator~$A$
in the $\mathrm{3D}$ rectangle
\[
\Omega =\mathrm R\coloneqq (0,L_1) \times (0,L_2) \times (0,L_3)
\]
under Lions boundary conditions, see~\eqref{LStokes}, where $L_1$, $L_2$, and~$L_3$ are positive real numbers. We follow the arguments
in~\cite[Section~3.5]{Phan-Thesis16}, where the case $L_1=L_2=L_3=\pi$ is considered. Notice that the vector length~$L=(L_1,L_2,L_3)$ plays a role
in the explicit computations, and different vector lengths may require slightly different arguments. Recall for example the case of a~$\mathrm{2D}$
rectangle~$(0,L_1) \times (0,L_2)$
considered in~\cite[Section~6.3]{Rod-Thesis08} where the case of a square~$L_1=L_2$ needs a particular consideration (see also~\cite{Rod06}). Recall also the case of the periodic
boundary conditions considered in~\cite[Section 2.3]{Shirikyan06} where in the case~$L_1=L_2=L_3$ it is possible to give an explicit form for
the $({\tt B},\D(A))$-saturating set (cf.~\cite[Remark~2.7]{Shirikyan06}, see also~\cite[Section~4]{Nersesyan10}).

\subsection{The saturating set}
We will present a saturating set for the rectangle under Lions boundary conditions, which consists of eigenfunctions of~$A$.

For a
given $k \in \N^3$, let $\#_0(k)$ stand for the number of vanishing components of $k$. A
complete system of eigenfunctions $\left\{  Y^{j(k),k}  \right\} $ is given by
\begin{subequations}\label{FamilyEig3DRect}
\begin{equation} \label{FamilyEig3DRect.Y}
 Y^{j\left(k\right),k} \coloneqq \begin{pmatrix}
w_1^{{j\left(k\right),k}} \sin \left( \frac{k_1 \pi x_1}{L_1} \right) \cos \left( \frac{k_2 \pi x_2}{L_2} \right) \cos \left( \frac{ k_3 \pi x_3}{L_3} \right)\\
w_2^{{j\left(k\right),k}} \cos \left( \frac{k_1 \pi x_1}{L_1} \right) \sin \left( \frac{k_2 \pi x_2}{L_2} \right) \cos \left( \frac{k_3 \pi x_3}{L_3} \right)\\ 
w_3^{{j\left(k\right),k}} \cos \left( \frac{k_1 \pi x_1}{L_1} \right) \cos \left( \frac{k_2 \pi x_2}{L_2} \right) \sin \left( \frac{k_3 \pi x_3}{L_3} \right)
 \end{pmatrix},\quad \#_0(k)\le1,
 \end{equation}
 with
 \begin{equation} 
\{ w ^{j(k),k}\mid j(k)\in\{1,2-\#_0(k)\}\}\subset\{k\}^{\perp_{[L]}}_0
\end{equation}
a linearly independent and orthogonal family and where
\begin{align}
\{k\}^{\perp_{[L]}}_0&\coloneqq\{z\in\R^3\setminus\{(0,0,0)\}\mid \left( z,\, k \right)_{[L]} =0,
\mbox{ and }z_i=0\mbox{ if }k_i=0\}, \\
(z,\,k)_{[L]} &\coloneqq \frac{ z_1 k_1}{L_1} + \frac{ z_2 k_2}{L_2} + \frac{z_3 k_3}{L_3}.
\end{align}
\end{subequations}
Notice that $2 - \#_0(k)$ is the dimension of the subspace $\{ k \}^{\perp_{[L]}}_0$ and that the orthogonality of the
family~$\{ w ^{j(k),k}\mid j(k)\in\{1,2-\#_0(k)\}\}$ implies that 
the family in~\eqref{FamilyEig3DRect.Y} is also orthogonal. The completeness of the system in~\eqref{FamilyEig3DRect.Y} is shown in \cite[Section 6.6]{PhanRod17}. 

\begin{example}
The eigenspace associated with a frequency vector~$k=(2,4,0)$, is the one spanned by the single eigenfunction
$ Y^{1,k}$, where we can choose~$w^{{1,k}}=C(-4L_1,2L_2,0)$ for any constant $C\ne0$. The eigenspace associated with a frequency
vector~$k=(2,4,5)\in\N_0^3$, is the one spanned by the
eigenfunctions $ Y^{1,k}$ and~$ Y^{2,k}$, where we can choose~$\{w^{1,k},w^{2,k}\}$ linearly independent in~$\linspan\{(-4L_1,2L_2,0),(-5L_1,0,2L_3)\}$.
\end{example}

Now we are able to present the saturating set in the following Theorem~\ref{T:satur3Drect}, whose proof is given in Section~\ref{sS:proofMainTh}. Before, we need
to derive some tools used in the proof.
\begin{theorem}\label{T:satur3Drect}
The set~$\CC \coloneqq\left\{ Y^{j(n),n} \left|\begin{array}{l}n\in\N^3,\quad 0 \le n_i \le 3,\\ \#_0(n) \le 1,\quad j(n) \in \{ 1, 2-\#_0(n) \}\end{array} \right. \right\}$
is $({\tt L},\D(A))$-saturating. 
\end{theorem}

\subsection{The expression for $\left( Y^k \cdot \nabla   \right) Y^m + \left( Y^m \cdot \nabla   \right) Y^k$}
Here we will present the expression for the coordinates of $\left( Y^{j(k),k} \cdot \nabla   \right) Y^{j(m),m} + \left( Y^{j(m),m} \cdot \nabla   \right) Y^{j(k),k}$
for given eigenfunctions as in~\eqref{FamilyEig3DRect.Y}. In order to shorten the following expressions and simplify the writing,
we will write
\[Y^{k}=Y^{j\left(k\right),k},\quad Y^{m}=Y^{j\left(m\right),m},\quad  w ^{k}=w ^{j(k),k},\quad  \mbox{and}\quad w ^{m}=w ^{j(m),m}\]
by omitting the indexes $j(k),j(m)$. We will also denote
\[
 \textstyle {\rm C}_i(k_i)\coloneqq \cos \left( \frac{k_i \pi x_i}{L_i} \right)\quad\mbox{and}\quad {\rm S}_i(k_i)\coloneqq \sin \left( \frac{k_i \pi x_i}{L_i} \right),\qquad i\in \{1,2,3\}.
\]

Using these notations, we find
\begin{equation*}
{\footnotesize
\left( Y^k \cdot \nabla   \right) Y^m = \textstyle\left( \begin{matrix}
Y^k \cdot w^m_1 \left( \ \begin{matrix}
\frac{m_1 \pi}{L_1} {\rm C}_1(m_1) {\rm C}_2(m_2) {\rm C}_3(m_3) \\
 -\frac{m_2 \pi}{L_2} {\rm S}_1(m_1) {\rm S}_2(m_2) {\rm C}_3(m_3) \\
 - \frac{m_3 \pi}{L_3} {\rm S}_1(m_1) {\rm C}_2(m_2) {\rm S}_3(m_3)
\end{matrix}   \right) \\
Y^k \cdot w^m_2 \left( \ \begin{matrix}
-\frac{m_1 \pi}{L_1} {\rm S}_1(m_1) {\rm S}_2(m_2) {\rm C}_3(m_3) \\
\frac{m_2 \pi}{L_2} {\rm C}_1(m_1) {\rm C}_2(m_2) {\rm C}_3(m_3) \\
 - \frac{m_3 \pi}{L_3} {\rm C}_1(m_1) {\rm S}_2(m_2) {\rm S}_3(m_3)
\end{matrix}   \right)\\
Y^k \cdot w^m_3 \left( \ \begin{matrix}
-\frac{m_1 \pi}{L_1} {\rm S}_1(m_1) {\rm C}_2(m_2) {\rm S}_3(m_3) \\
-\frac{m_2 \pi}{L_2} {\rm C}_1(m_1) {\rm S}_2(m_2) {\rm S}_3(m_3) \\
  \frac{m_3 \pi}{L_3}  {\rm C}_1(m_1) {\rm C}_2(m_2) {\rm C}_3(m_3)
\end{matrix}   \right)
\end{matrix}
\right),
}
\end{equation*}

To compute the coordinates of~$ \left( Y^k \cdot \nabla   \right) Y^m + \left( Y^m \cdot \nabla   \right) Y^k $ it will be useful to define
\begin{align} \label{BetaStar}
\beta^{\star_1\star_2\star_3}_{w^{k},m} \coloneqq \frac{\pi}{8} \textstyle\left( \star_1~\frac{w_1^{k} m_1}{L_1} \star_2~\frac{w_2^{k} m_2}{L_2} \star_3~\frac{w_3^{k}m_3 }{L_3}\right),
\quad\mbox{for}\quad (\star_1, \star_2, \star_3) \in \{+,- \}^3. 
\end{align}
As an illustration, we find the relations
$\beta^{+++}_{w^k,m} = \frac{\pi}{8} \left(\frac{w_1^{k} m_1}{L_1} + \frac{w_2^{k} m_2}{L_2} + \frac{w_3^{k} m_3}{L_3} \right)$, and
$\beta^{-+-}_{w^m,k} =   \frac{\pi}{8} \left(-\frac{w_1^{m} k_1}{L_1} + \frac{w_2^{m} k_2}{L_2} -\frac{w_3^{m} k_3}{L_3} \right)$.

From straightforward computations we can find
\begin{subequations}\label{CoorYY_Rect}
\begin{gather}
{\footnotesize
\begin{split} \label{1stCoor}
&\left( \left( Y^k \cdot \nabla   \right) Y^m + \left( Y^m \cdot \nabla   \right) Y^k \right)_1\\
&\hspace{3.9em}= +\left(w_1^{m}\beta^{+++}_{w^k,m} + w_1^{k} \beta^{+++}_{w^m,k} \right)
{\rm S}_1 (k_1 + m_1) {\rm C}_2(k_2 + m_2) {\rm C}_3 (k_3 + m_3)  \\
&\hspace{5em}+ \left( w_1^{m}\beta^{+++}_{w^k,m} - w_1^{k} \beta^{+++}_{w^m,k}\right)
{\rm S}_1 (k_1 - m_1) {\rm C}_2(k_2 - m_2) {\rm C}_3 (k_3 - m_3)  \\
&\hspace{5em}+ \left( w_1^{m}\beta^{++-}_{w^k,m}+ w_1^{k} \beta^{++-}_{w^m,k} \right)
{\rm S}_1 (k_1 + m_1) {\rm C}_2(k_2 + m_2) {\rm C}_3 (k_3 - m_3)  \\
&\hspace{5em}+  \left( w_1^{m}\beta^{++-}_{w^k,m} - w_1^{k} \beta^{++-}_{w^m,k} \right)
{\rm S}_1 (k_1 - m_1) {\rm C}_2(k_2 - m_2) {\rm C}_3 (k_3 + m_3)  \\
&\hspace{5em}+ \left( w_1^{m}\beta^{+-+}_{w^k,m} + w_1^{k} \beta^{+-+}_{w^m,k} \right)
{\rm S}_1 (k_1 + m_1) {\rm C}_2(k_2 - m_2) {\rm C}_3 (k_3 + m_3)  \\
&\hspace{5em}+ \left( w_1^{m}\beta^{+-+}_{w^k,m}- w_1^{k} \beta^{+-+}_{w^m,k} \right)
{\rm S}_1 (k_1 - m_1) {\rm C}_2(k_2 + m_2) {\rm C}_3 (k_3 - m_3)  \\
&\hspace{5em}+  \left( w_1^{m}\beta^{+--}_{w^k,m}+ w_1^{k} \beta^{+--}_{w^m,k} \right)
{\rm S}_1 (k_1 + m_1) {\rm C}_2(k_2 - m_2) {\rm C}_3 (k_3 - m_3)  \\
&\hspace{5em}+  \left( w_1^{m} \beta^{+--}_{w^k,m} - w_1^{k} \beta^{+--}_{w^m,k}\right)
{\rm S}_1 (k_1 - m_1) {\rm C}_2(k_2 + m_2) {\rm C}_3 (k_3 + m_3),
\end{split}
}
\end{gather}
\begin{gather}
{\footnotesize
\begin{split} \label{2ndCoor}
&\left(\left( Y^k \cdot \nabla   \right) Y^m + \left( Y^m \cdot \nabla   \right) Y^k  \right)_2\\
&\hspace{3.9em}= + \left( w_2^{m} \beta^{+++}_{w^k,m} + w_2^{k} \beta^{+++}_{w^m,k}\right)
{\rm C}_1 (k_1 + m_1) {\rm S}_2(k_2 + m_2) {\rm C}_3 (k_3 + m_3)  \\
&\hspace{5em}+\left( w_2^{m} \beta^{+++}_{w^k,m} - w_2^{k} \beta^{+++}_{w^m,k} \right)
{\rm C}_1 (k_1 - m_1) {\rm S}_2(k_2 - m_2) {\rm C}_3 (k_3 - m_3)  \\
&\hspace{5em}+  \left( w_2^{m} \beta^{++-}_{w^k,m} + w_2^{k} \beta^{++-}_{w^m,k} \right)
{\rm C}_1 (k_1 + m_1) {\rm S}_2(k_2 + m_2) {\rm C}_3 (k_3 - m_3)  \\
&\hspace{5em}+ \left( w_2^{m} \beta^{++-}_{w^k,m} - w_2^{k} \beta^{++-}_{w^m,k}  \right)
{\rm C}_1 (k_1 - m_1) {\rm S}_2(k_2 - m_2) {\rm C}_3 (k_3 + m_3)  \\
&\hspace{5em}+  \left( -w_2^{m} \beta^{+-+}_{w^k,m} + w_2^{k} \beta^{+-+}_{w^m,k} \right)
{\rm C}_1 (k_1 + m_1) {\rm S}_2(k_2 - m_2) {\rm C}_3 (k_3 + m_3)  \\
&\hspace{5em}+ \left( -w_2^{m} \beta^{+-+}_{w^k,m}- w_2^{k} \beta^{+-+}_{w^m,k} \right)
{\rm C}_1 (k_1 - m_1) {\rm S}_2(k_2 + m_2) {\rm C}_3 (k_3 - m_3)  \\
&\hspace{5em}+  \left( -w_2^{m} \beta^{+--}_{w^k,m}+ w_2^{k} \beta^{+--}_{w^m,k} \right)
{\rm C}_1 (k_1 + m_1) {\rm S}_2(k_2 - m_2) {\rm C}_3 (k_3 - m_3) \\
&\hspace{5em}+  \left( -w_2^{m} \beta^{+--}_{w^k,m}- w_2^{k} \beta^{+--}_{w^m,k} \right)
{\rm C}_1 (k_1 - m_1) {\rm S}_2(k_2 + m_2) {\rm C}_3 (k_3 + m_3), 
\end{split}
}
\end{gather}
\begin{gather}
{\footnotesize
\begin{split} \label{3rdCoor} 
&\left(\left( Y^k \cdot \nabla   \right) Y^m + \left( Y^m \cdot \nabla   \right) Y^k   \right)_3\\
&\hspace{3.9em}=+\left( w_3^{m} \beta^{+++}_{w^k,m} + w_3^{k}  \beta^{+++}_{w^m,k} \right)
{\rm C}_1 (k_1 + m_1) {\rm C}_2(k_2 + m_2) {\rm S}_3 (k_3 + m_3)  \\
&\hspace{5em}+\left( w_3^{m} \beta^{+++}_{w^k,m}  - w_3^{k}  \beta^{+++}_{w^m,k} \right)
{\rm C}_1 (k_1 - m_1) {\rm C}_2(k_2 - m_2) {\rm S}_3 (k_3 - m_3) \\
&\hspace{5em}+ \left( -w_3^{m} \beta^{++-}_{w^k,m}  + w_3^{k}  \beta^{++-}_{w^m,k}\right)
{\rm C}_1 (k_1 + m_1) {\rm C}_2(k_2 + m_2) {\rm S}_3 (k_3 - m_3) \\
&\hspace{5em}+ \left( -w_3^{m} \beta^{++-}_{w^k,m}  - w_3^{k}  \beta^{++-}_{w^m,k} \right)
{\rm C}_1 (k_1 - m_1) {\rm C}_2(k_2 - m_2) {\rm S}_3 (k_3 + m_3) \\
&\hspace{5em}+ \left( w_3^{m} \beta^{+-+}_{w^k,m} + w_3^{k}  \beta^{+-+}_{w^m,k} \right)
{\rm C}_1 (k_1 + m_1) {\rm C}_2(k_2 - m_2) {\rm S}_3 (k_3 + m_3)  \\
&\hspace{5em}+ \left( w_3^{m} \beta^{+-+}_{w^k,m}  - w_3^{k}  \beta^{+-+}_{w^m,k} \right)
{\rm C}_1 (k_1 - m_1) {\rm C}_2(k_2 + m_2) {\rm S}_3 (k_3 - m_3)  \\
&\hspace{5em}+ \left( -w_3^{m} \beta^{+--}_{w^k,m} + w_3^{k}  \beta^{+--}_{w^m,k} \right)
{\rm C}_1 (k_1 + m_1) {\rm C}_2(k_2 - m_2) {\rm S}_3 (k_3 - m_3) \\
&\hspace{5em}+ \left( -w_3^{m} \beta^{+--}_{w^k,m} - w_3^{k}  \beta^{+--}_{w^m,k} \right)
{\rm C}_1 (k_1 - m_1) {\rm C}_2(k_2 + m_2) {\rm S}_3 (k_3 + m_3) .
\end{split}
}
\end{gather}
\end{subequations}

Accordingly to Definition~\ref{D:satur-l},  we would need to compute the
orthogonal projection~$\BB(Y^k,Y^m)=\Pi\left(\left( Y^k \cdot \nabla   \right) Y^m + \left( Y^m \cdot \nabla   \right) Y^k\right)$, onto~$H$.
However, we will manage to use only the
coordinates in~\eqref{CoorYY_Rect} instead of the explicit expression for $\BB(Y^k,Y^m)$ (cf. Section~\ref{sS:remprojBB}).
The expression for $\BB(Y^k,Y^m)$ can be more
cumbersome than the expressions in~\eqref{CoorYY_Rect}. For the case~$L=(L_1,L_2,L_3)=(\pi,\pi,\pi)$,
the explicit expression for~$\BB(Y^k,Y^m)$ can be found in~\cite[Section~3.5.1]{Phan-Thesis16}.

\subsection{A difference between $\mathrm{2D}$ and $\mathrm{3D}$ cases}
For the case of $\mathrm{2D}$ Navier--Stokes equation on a rectangle under Lions boundary conditions, treated in~\cite{Rod06}, it holds that $B(W^n,W^n)=0$
for an eigenfunction~$W^n$ of the corresponding $\mathrm{2D}$ Stokes operator (cf.~\cite[Section~4.5]{Rod-Thesis08}). This can be seen from the fact that vectors fields in~$u\in H$
can be identified with a so-called stream function~$\phi_u$, as~$u=\nabla^\perp\phi_u$, and that we have the vorticity relations $\nabla^\perp\cdot u=-\Delta\phi_u$ and
$\nabla^\perp\cdot B(u,u)=-u\cdot\nabla(\nabla^\perp\cdot u)=\nabla\phi_u\cdot\nabla^\perp(\nabla^\perp\cdot u)=-\nabla\phi_u\cdot\Delta u$.
Thus~$\nabla^\perp\cdot B(W^n,W^n)=\lambda_n\nabla\phi_{W^n}\cdot W^n=\lambda_n\nabla\phi_{W^n}\cdot \nabla^\perp\phi_{W^n}=0$, where $\lambda_n$ is the eigenvalue associated
to~$W^n$, $\Pi(-\Delta) W_n=\lambda_n W_n$.

From Theorem~\ref{T:diff2D3D} below, in the case of the $\mathrm{3D}$ rectangle, the identity $B(Y^k,Y^k)=0$ does not hold for all eigenfunctions~$Y^k$ (cf.~the case of
the 1D Burgers equation studied in~\cite{PhanRod-ecc15}). 
\begin{theorem}\label{T:diff2D3D}
For an eigenfuntion~$Y^k=Y^{j(k),k}$ as in~\eqref{FamilyEig3DRect}, we have
 \begin{align*}
B(Y^k,Y^k)&\ne0,\qquad\mbox{if}\quad \#_0(k)=0,\\
B(Y^k,Y^k)&=0,\qquad\mbox{if}\quad \#_0(k)=1.
\end{align*}
\end{theorem}
\begin{proof}
Indeed in the case $\#_0(k)=0$, since $0=(w_1^{k},k)_{[L]}=\frac{w_1^{k} k_1}{L_1} +~\frac{w_2^{k} k_2}{L_2} +~\frac{w_3^{k} k_3}{L_3}$,
from~\eqref{CoorYY_Rect} with $m = k$, we can rewrite the first coordinate in short form as follows 
\begin{gather*} 
\begin{split}
&\quad\left( \left( Y^k \cdot \nabla   \right) Y^k \right)_1\\
& = -\frac{\pi}{4} w_1^k \fractx{w_3^k k_3}{L_3}~{\rm S}_1(2k_1) {\rm C}_2(2k_2)
- \frac{\pi}{4} w_1^k \fractx{w_2^k k_2}{L_2}~{\rm S}_1(2k_1) {\rm C}_3(2k_3) +  \frac{\pi}{4} w_1^k \fractx{w_1^k k_1}{L_1}~{\rm S}_1(2k_1)\\
& = -\frac{\pi}{2}  w_1^k \sin\left(\fractx{2k_1 \pi x_1}{L_1} \right) \left( \fractx{w_3^k k_3}{L_3}  \cos^2\left(\fractx{k_2 \pi x_2}{L_2} \right)  +  \fractx{w_2^k k_2}{L_2}  \cos^2 \left( \fractx{k_3 \pi x_3}{L_3} \right) \right).
\end{split}
\end{gather*} 
Proceeding analogously for the other two coordinates, we obtain 

\begin{align}
 \label{Casekeqm}{\footnotesize
  \left( Y^k \cdot \nabla   \right) Y^k  = -\frac{\pi}{2}  \begin{pmatrix}
 w_1^k \sin\left(\frac{2k_1 \pi x_1}{L_1} \right) \left( \frac{w_3^k k_3}{L_3}  \cos^2\left(\frac{k_2 \pi x_2}{L_2} \right)  +  \frac{w_2^k k_2}{L_2}  \cos^2 \left( \frac{k_3 \pi x_3}{L_3} \right) \right)\\
w_2^k \sin\left(\frac{2k_2 \pi x_2}{L_2} \right) \left( \frac{w_3^k k_3}{L_3}   \cos^2\left(\frac{k_1 \pi x_1}{L_1} \right)  +  \frac{w_1^k k_1}{L_1} \cos^2 \left( \frac{k_3 \pi x_3}{L_3} \right) \right)  \\
    w_3^k \sin \left(\frac{2k_3 \pi x_3}{L_3} \right) \left( \frac{w_1^k k_1}{L_1}  \cos^2\left(\frac{k_2 \pi x_2}{L_2} \right)    +  w_2^k k_2  \cos^2\left(\frac{k_1 \pi x_1}{L_1} \right) \right) 
  \end{pmatrix}.
  }
  \end{align}
Assuming that $B(Y^{k},Y^k)=\Pi\left(\left( Y^k \cdot \nabla   \right) Y^k\right)= 0$, there would exist a function $g$ such that $  \left( Y^k \cdot \nabla   \right) Y^k  = \nabla g$
because $H^\bot = \{ \nabla g \mid g \in H^1 (\Omega, \R)\}$ (cf. \cite[Section 2.5]{Temam95}),
which implies that $\curl \left( \left( Y^k \cdot \nabla   \right) Y^k \right) = \curl (\nabla g) = 0$. That is, 
\begin{align} \label{Curl}{\footnotesize
0 = \curl \left( \left( Y^k \cdot \nabla   \right) Y^k \right)= \frac{\pi^2}{2}  \begin{pmatrix}
\frac{w_1^k k_1}{L_1} \sin\left(\frac{2k_2 \pi x_2}{L_2} \right)  \sin \left(\frac{2k_3 \pi x_3}{L_3} \right) \left( \frac{w_3^k~k_2}{L_2}~-~\frac{w_2^k k_3}{L_3}  \right) \\
\frac{w_2^k k_2}{L_2}  \sin \left(\frac{2k_3 \pi x_3}{L_3} \right) \sin\left(\frac{2k_1 \pi x_1}{L_1} \right) \left( \frac{w_1^k~k_3}{L_3}~-~\frac{w_3^k k_1}{L_1}  \right) \\
 \frac{w_3^k k_3}{L_3} \sin\left(\frac{2k_1 \pi x_1}{L_1} \right) \sin\left(\frac{2k_2 \pi x_2}{L_2} \right) \left( \frac{w_2^k~k_1}{L_1}~-~\frac{w_1^k k_2}{L_2}  \right)
\end{pmatrix}.
}
\end{align} 
We will prove that this equality cannot hold if~$\#_0(k)=0$. We start by proving that, in this case, no component of $w^k$ is vanishing.
Indeed, if for example $w_1^k = 0$, we would have $\frac{w_2^k k_2}{L_2}  = - \frac{w_3^k k_3}{L_3}$. Then,~\eqref{Curl} would give us 
\begin{align*}
0 &= \begin{pmatrix}
0 \\
  - \frac{w_2^k k_2}{L_2} \fractx{w_3^k k_1}{L_1}\sin \left(\fractx{2k_3 \pi x_3}{L_3} \right) \sin\left(\frac{2k_1 \pi x_1}{L_1} \right) \\
   \frac{w_3^k k_3}{L_3}  \fractx{w_2^k k_1}{L_1} \sin\left(\fractx{2k_1 \pi x_1}{L_1} \right) \sin\left(\frac{2k_2 \pi x_2}{L_2} \right)
\end{pmatrix} \\ &=  \fractx{w_3^k k_3}{L_3} \fractx{k_1}{L_1} \sin\left(\fractx{2k_1 \pi x_1}{L_1} \right)
 \begin{pmatrix}
0 \\
w_3^k \sin \left(\fractx{2k_3 \pi x_3}{L_3} \right) \\
   w_2^k   \sin\left(\fractx{2k_2 \pi x_2}{L_2} \right)
\end{pmatrix}. 
\end{align*}
Since $k \in \N_0^3$, it follows that necessarily~$(w_3^k)^2=w_2^kw_3^k=0$, which in turn leads us to $w^k =(0,0,0)$.
This contradicts the fact that by the definition $w^k\ne0$, because the family $\{w^{j(k),k} \mid j(k) \in \{1, 2-\#_0(k) \} \}$ must
be linearly independent\index{linearly independent}. Thus $w_1^k \neq 0$. A similar argument leads us to $w_2^k \neq 0$ and $w_3^k \neq 0$.

Now, since all components of $w^k$ are different from 0, from~\eqref{Curl}, we have 
\begin{align*}
 \frac{w_3^k~k_2}{L_2}~-~\frac{w_2^k k_3}{L_3} = \frac{w_1^k~k_3}{L_3}~-~\frac{w_3^k k_1}{L_1}   = \frac{w_2^k~k_1}{L_1}~-~\frac{w_1^k k_2}{L_2}  = 0, 
\end{align*}
that is $k^L \times w^k = 0$, with~$k^L\coloneqq (\frac{k_1}{L_1},\frac{k_2}{L_2},\frac{k_3}{L_3})$. Furthermore $w^k \cdot k^L = (w^k,k)_{[L]}= 0$ and from the
triple vector product relation
\begin{align*}
k^L \times \left( k^L \times w^k \right) = (k^L \cdot w^k) k^L - \left(k^L \cdot k^L \right) w^k,
\end{align*}
(cf.~\cite[Section~2.35]{Aris89}) it follows that $0=0-\left( k^L \cdot k^L \right) w^k = -\norm{k^L}{}^2w^k$ which leads to the
contradiction $w^k = 0$. Therefore we can conclude that $B(Y^{k},Y^{k}) \neq 0$ for all $k \in \N_0^3$. 

In the case $\#_0(k)=1$, for example if~$k_3=0$, then~$w^k_3=0$ and from~\eqref{Casekeqm} we obtain
\begin{align*}
   \left( Y^k \cdot \nabla   \right) Y^k  = -\frac{\pi}{2}  \begin{pmatrix}
 w_1^k \sin\left(\frac{2k_1 \pi x_1}{L_1} \right) \frac{w_2^k k_2}{L_2}  \\
w_2^k \sin\left(\frac{2k_2 \pi x_2}{L_2} \right) \frac{w_1^k k_1}{L_1}  \\
   0
  \end{pmatrix}=\nabla g
\end{align*}
with $g=\frac{w_1^kw_2^k}{4}\left(\frac{k_2 L_1}{k_1L_2}\cos\left(\frac{2k_1 \pi x_1}{L_1} \right)+\frac{k_1L_2}{k_2L_1}\cos\left(\frac{2k_2 \pi x_2}{L_2} \right)\right)$.
Thus $B(Y^{k},Y^{k}) = 0$, if $k_3=0$. A similar argument gives us that~$B(Y^{k},Y^{k})  = 0$ when~$k_i=0$, for~$i\in\{1,2\}$.
\null\hfill\null{\qed}\end{proof}

\subsection{Avoiding the computation of $\BB(Y^k,Y^m)$}\label{sS:remprojBB}
We present here an auxiliary result which will allow us to work with the coordinates in~\eqref{CoorYY_Rect}, avoiding to derive (and avoiding the need to work with) the
explicit expression for the projection~$\BB(Y^k,Y^m)=\Pi\left(\langle Y^k\cdot\nabla\rangle Y^m+\langle Y^m\cdot\nabla\rangle Y^k\right)$ (cf.~Definition~\ref{D:satur-l}).

With $k \in \N^3$, let us define the functions 
\begin{align*}
\psi_{1}^k  &= \psi_{1}^k (x) = \sin (\textstyle\frac{k_1\pi x_1}{L_1}) \cos (\textstyle\frac{k_2\pi x_2}{L_2}) \cos (\textstyle\frac{k_3\pi x_3}{L_3}), \\
\psi_{2}^k  &= \psi_{2}^k (x) = \cos (\textstyle\frac{k_1\pi x_1}{L_1}) \sin (\textstyle\frac{k_2\pi x_2}{L_2}) \cos (\textstyle\frac{k_3\pi x_3}{L_3}), \\
\psi_{3}^k  &= \psi_{3}^k (x) = \cos (\textstyle\frac{k_1\pi x_1}{L_1}) \cos (\textstyle\frac{k_2\pi x_2}{L_2}) \sin (\textstyle\frac{k_3\pi x_3}{L_3}), 
\end{align*}
and the vector functions 
\begin{equation}\label{cal_YY}
\YY_z^k = \begin{pmatrix}
z_1 \psi_{1}^k \\ z_2 \psi_{2}^k \\ z_3 \psi_{3}^k 
\end{pmatrix}, \qquad k \in \N^3,~~ z \in \R^3.
\end{equation}
we observe that for the eigenfunctions~$Y^k=Y^{j(k), k}$ in~\eqref{FamilyEig3DRect.Y}, we have 
\begin{align*}
Y^{j(k), k} = \YY^k_{w^{j(k).k}} \text{~with}~~ k \in \N^3,~~ \#_0(k) \le 1,~~ j(k) \in \{1, 2- \#_0(k) \} . 
\end{align*}
Observe also that if $m \neq k$ then $\left( \YY_z^k, \YY_w^m  \right)_{L^2(\Omega, \R^3)} = 0  $ for all $z, w \in \R^3$,
because we have $\left( \psi_i^k, \psi_i^m   \right)_{L^2((0,L_i),R)} = 0$ for all $i \in \{1,2,3 \}$. 
From~\eqref{CoorYY_Rect}, we observe that 
\begin{align} \label{new-form}
\left( Y^k \cdot \nabla   \right) Y^m + \left( Y^m \cdot \nabla   \right) Y^k &= \sum_{\begin{subarray}{l} n = (k (\star_1 \star_2 \star_3) m)^+ \\
(\star_1,\star_2, \star_3) \in \{-,+\}^3
 \end{subarray}}{\YY^n_{z^n}},
\end{align}
where
\begin{align}
 (k (\star_1 \star_2 \star_3) m)^+&\coloneqq (|k_1\star_1 m_1|,|k_2\star_2 m_2|,|k_3\star_3 m_3| ),
\end{align}
and for suitable vectors $z^n=(z^n_1,z^n_2,z^n_3) \in \R^3$ (depending on the parameters $k$, $m$, $w^m$ and $w^k$).  Thus the projection 
\begin{align*}
\BB (Y^m, Y^k)  =  \sum_{\begin{subarray}{l} n = (k (\star_1 \star_2 \star_3) m)^+ \\
(\star_1,\star_2, \star_3) \in \{-,+\}^3 \\
\#_0(n) \le 1, \\
j(n) \in \{1, 2- \#_0(n)\}
 \end{subarray}}{\alpha^{j(n),n} Y^{j(n),n}}
\end{align*}
satisfies, for any ~$n$, 
\begin{align*}
\sum_{j(n) \in \{1, 2- \#_0(n)\}} {\alpha^{j(n),n} Y^{j(n),n} } = \Pi \YY^n_{z^n}  = \Pi \begin{pmatrix}
z^n_1 \psi_1^n \\  z^n_2 \psi_2^n \\ z^n_3 \psi_3^n
\end{pmatrix}.
\end{align*}
\begin{lemma} \label{lemma-lin-indp}
Let us be given $\alpha, \gamma \in \R^3$ and $k \in \N_0^3$. Then the family $\{ \alpha, \gamma, k\}$ is linearly
independent if, and only if,  the family  $\{ \Pi \YY_\alpha^k,   \Pi \YY_\gamma^k \} $ is linearly independent. In either case
\[\linspan\{ \Pi \YY_\alpha^k,   \Pi \YY_\gamma^k \} =\linspan Y^{\{1,2\},k}.\]
\end{lemma}

\begin{proof}
Let us fix a basis\index{basis} $\{  w^{1,k}, w^{2,k}  \}$ for $\{k\}_0^{\bot} = \{k\}^{\bot}$. 

Given $\alpha, \gamma \in \R^3$, since $\{w^{1,k}, w^{2,k}, k\}$ is a basis in $\R^3$, we can write (in an unique way)
\begin{align} \label{ChangeBasis}
\begin{split}
\alpha &= \alpha^{1,k} w^{1,k} + \alpha^{2,k} w^{2,k} + \alpha_0 k, \\
\gamma  &= \gamma^{1,k} w^{1,k} + \gamma^{2,k} w^{2,k} + \gamma_0 k.
\end{split}
\end{align}
and it follows that 
\begin{align*}
\YY^k_{\alpha} &= \alpha^{1,k} Y^{1,k} + \alpha^{2,k}  Y^{2,k} + \alpha_0 \YY_k^{k}, \\
\YY^k_{\gamma} &=  \gamma^{1,k} Y^{1,k} + \gamma^{2,k} Y^{2,k} + \gamma_0 \YY_k^{k}.
\end{align*} 
Since $ \YY_k^{k} = \nabla (- \cos (\frac{k_1 \pi x_1}{L_1}) \cos (\frac{k_2\pi x_2}{L_2}) \cos(\frac{k_3\pi x_3}{L_3}))$,  we obtain 
\begin{equation} \label{Comb-Proj}
\begin{split}
 \Pi \YY^k_{\alpha} =  \alpha^{1,k} Y^{1,k} +  \alpha^{2,k}  Y^{2,k}, \\ 
\Pi \YY^k_{\gamma} =  \gamma^{1,k} Y^{1,k} +  \gamma^{2,k}  Y^{2,k}.
\end{split}
\end{equation}
Now, it is clear that~$\linspan\{ \Pi \YY_\alpha^k,   \Pi \YY_\gamma^k \} =\linspan Y^{\{1,2\},k}$ if, and only if, the
family~$\{ \Pi \YY_\alpha^k,   \Pi \YY_\gamma^k \}$ is linearly independent. Recall that~$\{Y^{1,k},Y^{2,k}\}$ is linearly independent by definition.

Observe that given $(r, s) \in \R^2$ such that
$r \Pi \YY^k_{\alpha} + s \Pi \YY^k_{\gamma} = 0$, we have (using \eqref{Comb-Proj})
that $(r \alpha^{1,k} + s \gamma^{1,k})  Y^{1,k}  + (r \alpha^{2,k} + s \gamma^{2,k})  Y^{2,k} = 0$ and, since $\{ Y^{1,k}, Y^{2,k}\}$ is linearly
independent, we find that
$\begin{pmatrix}
\alpha^{1,k} & \gamma^{1,k} \\
\alpha^{2,k} & \gamma^{2,k}
\end{pmatrix} 
\begin{pmatrix}
r \\
s 
\end{pmatrix}  
= \begin{pmatrix}
0 \\
0 
\end{pmatrix}.
$
Therefore 
\begin{equation} \label{Pilinind}
\{ \Pi \YY_\alpha^k,   \Pi \YY_\gamma^k\}\;\mbox{ is linearly independent if, and only if, }\;\det \begin{pmatrix}
 \alpha^{1,k} & \alpha^{2,k}  \\
  \gamma^{1,k} &\gamma^{2,k} 
\end{pmatrix} \neq 0. 
\end{equation}

Since $\{w^{1,k}, w^{2,k}, k\}$ is linearly independent, a similar argument (using~\eqref{ChangeBasis} together with $k=0w^{1,k}+0w^{2,k}+1k$) leads us to
\begin{equation} \label{Pilinind2}
\{ \alpha, \gamma, k \}\;\mbox{ is linearly independent if, and only if, }\;\det \begin{pmatrix}
 \alpha^{1,k} & \alpha^{2,k}  &  \alpha_0 \\
  \gamma^{1,k} &\gamma^{2,k} & \gamma_0 \\
  0 & 0 & 1
\end{pmatrix} \neq 0. 
\end{equation}
The Lemma follows from~\eqref{Pilinind} and~\eqref{Pilinind2}. 
\null\hfill\null{\qed}\end{proof}

\subsection{Proof of Theorem~\ref{T:satur3Drect}}\label{sS:proofMainTh}
Introducing the family of sets
\begin{equation}\label{setsSC_R}
\begin{array}{l}
\sS^q  \coloneqq \left\{ n\in\N^3 \mid 0 \le n_i \le q,~ \#_0(n) \le 1 \right\},\\
\CC^q  \coloneqq \left\{ Y^{j(n),n} \mid~n \in \sS^q ,~ j(n) \in \{ 1, 2-\#_0(n) \}  \right\},\end{array}\qquad q\in\N ,\quad q\ge 3 ,
\end{equation}
and recalling the sequence in Definition~\ref{D:satur-l}, we  can see that Theorem~\ref{T:satur3Drect} is a corollary of the following inclusions
\begin{equation}\label{Sat_Inc_R}
 \CC^3 \subseteq \GG^{0},\quad\mbox{and}\quad\CC^q \subseteq \GG^{q-1},\qquad \mbox{for all}\quad q\in\N,\quad q\ge3,
\end{equation}
which we will prove by induction.

\noindent{\bf Base step.}
By definition, $\CC=\CC^3 $ and~$\linspan\CC=\GG^0\subseteq \GG^{2}$. Therefore
\begin{equation}
 \mbox{Inclusions~\eqref{Sat_Inc_R} holds for }q=3. \label{BaseIndStep_R}
\end{equation}
\noindent{\bf Induction step.} The induction hypothesis is
\addtocounter{equation}{1}
\begin{equation}\label{IH.R}
 \CC^3 \subseteq \GG^{0} \mbox{ and the inclusion $\CC^q \subseteq \GG^{q-1}$ holds true for a given $q\in\N$, $q\ge3$}.\tag{{\bf IH.R}-eq.\theequation}
\end{equation}
We want to prove that~$\CC^{q+1} \subseteq \GG^{q}$. \\

 Notice that
\[
\CC^{q+1}  \coloneqq\left\{ Y^{1,n} \mid~n \in \sS^{q+1} ,~ \#_0(n)=1\right\}\bigcup\left\{ Y^{1,n},Y^{2,n} \mid~n \in \sS^{q+1} ,~ \#_0(n)=0 \}  \right\},
\]
We will consider the cases~$\#_0(n)=1$ and~$\#_0(n)=0$ separately.

{\em $\bullet$ The case~$n\in\CC^{q+1} $ and~$\#_0(n)=1$}.
Suppose that $k\in \N^3$, $\#_0(k)=1$, and~$k_3=0$. We can see that, up to a constant $C\ne0$,
$Y^k= C\begin{pmatrix}
W_{\overline k}\\0\end{pmatrix}$, where for simplicity we denoted~$Y^{k}=Y^{1,k}$ and~$W_{\overline k}\coloneqq\begin{pmatrix}
\textstyle\frac{-k_2\pi}{L_2}{\rm S}_1(k_1){\rm C}_2(k_2)\\
\textstyle\frac{k_1\pi}{L_1}{\rm C}_1(k_1){\rm S}_2(k_2)\end{pmatrix}$, with $\overline k\coloneqq(k_1,k_2)$. Notice
that~$W_{\overline k}$ is an eigenfunction of the Stokes operator in the $\mathrm{2D}$ rectangle~$\mathrm R_2=(0,L_1)\times(0,L_2)$, as observed
in~\cite[Section~2.2]{Rod06}. Now let also $m\in \N^3$, $\#_0(m)=1$, and~$m_3=0$. Then, we can see that
\begin{subequations}\label{Proj_RR2}
\begin{equation}\label{Proj0}
 \left( \left( Y^k \cdot \nabla   \right) Y^m + \left( Y^m \cdot \nabla   \right) Y^k \right)
 =\begin{pmatrix}\left( \left( W_{\overline k} \cdot \nabla_2   \right) W_{\overline m} + \left( W_{\overline m} \cdot \nabla_2   \right) W_{\overline k} \right)\\0\end{pmatrix}
\end{equation}
 where~$\nabla_2$ is the gradient on the rectangle~$\mathrm R_2$, that is, on the variables~$(x_1,x_2)$.
 
 Now, on one hand we can write
  \begin{equation}\label{ProjR}
 \left( \left( Y^k \cdot \nabla   \right) Y^m + \left( Y^m \cdot \nabla   \right) Y^k \right)=\BB(Y^k, Y^m)+\nabla q
 \end{equation}
 where~$\BB(Y^k, Y^m)\in H$ and~$q\in H^1(\mathrm R,\R^3)$. On the other hand we can write
 \begin{equation}\label{ProjR2}
 \left( W_{\overline k} \cdot \nabla_2   \right) W_{\overline m} + \left( W_{\overline m} \cdot \nabla_2   \right) W_{\overline k}=\BB_2(W_{\overline k}, W_{\overline m})+\nabla p
 \end{equation} 
 \end{subequations}
 where~$\BB_2(Y^k, Y^m)\in\{u\in L^2(\mathrm R_2,\R^2)\mid \p_{x_1}u_1+\p_{x_2}u_2=0 \mbox{ and }u\cdot(\nnn_1,\nnn_2)=0\}$ and~$p\in H^1(\mathrm R_2,\R^2)$.
  Therefore from~\eqref{Proj_RR2} it follows that necessarily
 \[
 \BB(Y^k, Y^m)=\begin{pmatrix}\BB_2(W_{\overline k}, W_{\overline m})\\0\end{pmatrix}\quad\mbox{and}\quad\nabla q=\begin{pmatrix}\nabla_2 p\\0\end{pmatrix}. 
 \]

 Notice that given~$x\in\p\mathrm R$, the normal~$\nnn_{x}$, to~$\mathrm R$  at~$x$, satisfies \\$\nnn_{x}=(\nnn_{x,1},\nnn_{x,2},\nnn_{x,3})=(\nnn_{x,1},\nnn_{x,2},0)$ if
 $x\eqqcolon(\overline x,x_3)\in\p\mathrm R_2\times(0,L_3)$, and $\nnn_{x}=(0,0,\pm 1)$ if~$x\in\mathrm R_2\times\{0,\,L_3\}$. Notice also that
 $\p\mathrm R_2\times(0,L_3)\bigcup\mathrm R_2\times\{0,\,L_3\}$ is dense in~$\p\mathrm R$.

From the results in~\cite[Section~6.3]{Rod-Thesis08} (see also~\cite[Section 7.1]{Rod06}, for $({\tt B},\D(A))$-saturating sets) we know that
if for all $q\ge3$ and $n\in\sS^{q+2} $,
with~$n_3=0$ and~$(n_1,n_2)\ne(q+2,q+2)$,
we have that~$W_{\overline n}\in\GG^{q+2-3+1}$, then for all $n\in\sS^{q+1} $, with $n_3=0$, we have that~$W_{\overline n}\in\GG^{q}$. Repeating the argument for the cases
$n_1=0$ and $n_2=0$, we arrive at
\begin{equation}\label{one0_R}
 \linspan\{Y^n\mid n\in\sS^{q+1} , \#_0(n)=1\}\subseteq \GG^{q}.
\end{equation}

{\em $\bullet$ The case~$n\in\CC^{q+1} $ and~$\#_0(n)=0$}. In this case~$n\in\N^3_0$.
We start by defining, again for~$q\ge 3$ and for some given~$m$, $m_1$, and~$m_2$ in~$\{1,2,3\}$, the index sets
%
\begin{equation}\label{setsRL_R}
\begin{split}
\RR^q_m &\coloneqq \left\{ n \in \sS^q   \mid n_m = q,~1 \le n_i \le q-1~\mbox{for }i \neq m \right\},\\
\LL^q_{m_1,m_2}& \coloneqq \left\{ n \in \sS^q    \mid ~n_{m_1} = q= n_{m_2},~m_1\ne m_2, 1 \le n_i \le q-1,~i \notin \{m_1,m_2\} \right\}.
\end{split}
\end{equation}

We define the set of eigenfunctions
\begin{align*}
\CC^q_{0} = \left\{ Y^{1,n},Y^{2,n} \mid~n \in \sS^q ,~\#_0(n)=0 \right\}.
\end{align*}
Notice that
\begin{equation} \label{subset-sq-2Drect}
 \begin{split}
\{n\in\sS^{q+1} \mid\#_0(n)=0\}  &=  \{n\in\sS^{q} \mid\#_0(n)=0\} {\textstyle\bigcup} \left( {\RR^{q+1}_1}\cup {\RR^{q+1}_2}\cup{\RR^{q+1}_3} \right)\\
&\quad{\textstyle\bigcup} \left( \LL^{q+1}_{1,2} \cup \LL^{q+1}_{2,3} \cup \LL^{q+1}_{3,1} \right) {\textstyle\bigcup} \{ (q+1, q+1, q+1) \}.
\end{split}
\end{equation}

It remains to prove that~$\CC^{q+1}_{0} \subset\GG^{q}$, which is a corollary of the following Lemmas~\ref{L:Part1_R}, \ref{L:Part2_R}, and~\ref{L:Part3_R}
which we will prove in the following Sections~\ref{ssS:Part1_R},~\ref{ssS:Part2_R},
and~\ref{ssS:Part3_R}.

\begin{lemma}\label{L:Part1_R}
 $Y^{j(n),n}\in\GG^{q}$ for all $n \in \bigcup \limits_{i = 1}^3 {\RR^{q+1}_i}$.
\end{lemma}

\begin{lemma}\label{L:Part2_R}
$Y^{j(n),n}\in\GG^{q}$ for all $n \in \LL^{q+1}_{1,2} \cup \LL^{q+1}_{2,3} \cup \LL^{q+1}_{3,1}$. 
\end{lemma}

\begin{lemma}\label{L:Part3_R}
 $Y^{\{1,2\},(q+1,q+1,q+1)}\subset\GG^{q}$.
\end{lemma}

Observe that, from~\eqref{one0_R} and Lemmas~\ref{L:Part1_R}, \ref{L:Part2_R}, and~\ref{L:Part3_R},  it follows that
\begin{align} \label{res-final-3Drect}
Y^{j(n),n} \in \GG^{q} \quad \text{for all}\quad n \in \sS^{q+1} .
\end{align}
which implies that~$\CC^{q+1} \subseteq\GG^{q}$.
Therefore, we have just proven that~\eqref{IH.R} implies that~$\CC^{q+1} \subseteq\GG^{(q+1)-1}$.
Then by induction, using~\eqref{BaseIndStep_R}, it follows that~\eqref{Sat_Inc_R}
holds true, which implies the statement of Theorem~\ref{T:satur3Drect}.\null\hfill\null{\qed}

\subsubsection{Proof of Lemma~\ref{L:Part1_R}}\label{ssS:Part1_R}
We proceed into \ref{Step3_3Drect} main steps:
\begin{enumerate}
\renewcommand{\theenumi}{{\sf\arabic{enumi}}} 
 \renewcommand{\labelenumi}{} 
\item $\bullet$~Step~\theenumi:\label{Step2_3Drect} {\em~Generating  $Y^{j(n),n}$ with $n \in\{(1,l, q+1),(l,1, q+1)\mid 0<l\le q\}$.}
\item $\bullet$~Step~\theenumi:\label{Step3_3Drect} {\em~ Generating $Y^{j(n),n}$ with $n\in\{(n_1,n_2, q+1)\mid 2 \le n_1 \le q \mbox{ and }2 \le n_2 \le q\}$.}
\end{enumerate}

\bigskip\noindent
{$\bullet$~Step~\ref{Step2_3Drect}: {\em~Generating the family $Y^{j(n),n}$ with $n = (1,l, q+1)$ or $n = (l,1, q+1)$.} We start with~$n = (1,l, q+1)$ and proceed by induction on~$l$.

\noindent{\bf Base step.} We will prove that
\begin{equation}\label{11_ok_R}
Y^{\{1,2 \},(1,1, q+1)} \subset  \GG^{q}.
\end{equation}

To generate~$n=(1,1, q+1)$ we choose
\begin{align*}
k &=(1,0,q),& m &=(0,1,1),\\ w^k &= (L_1q,0,-L_3),& w^m &= (0,L_2,-L_3),
\end{align*}

From~\eqref{CoorYY_Rect}, this choice gives us 
\begin{align*}
\left( Y^k \cdot \nabla   \right) Y^m + \left( Y^m \cdot \nabla   \right) Y^k = \YY^{(1,1,q+1)}_{z_{\alpha^1}} + \YY^{(1,1,q-1)}_{z_{\alpha^2}}, 
\end{align*}
for suitable $z_{\alpha^1}, z_{\alpha^2} \in \R^3$.
By the induction hypothesis in assumption~\\\eqref{IH.R}, we have $Y^{\{1,2 \},(1,1,q-1)}\coloneqq\{Y^{1 ,(1,1,q-1)},Y^{2 ,(1,1,q-1)}\} \subseteq \GG^{q-1}
\subseteq \GG^{q}$, which implies that~$\Pi \YY^{(1,1,q-1)}_{z^{\alpha^2}} \in \GG^{q}$.  Hence, we can conclude that
$\Pi \YY^{(1,1,q+1)}_{z^{\alpha^1}} \in \GG^{q}$. 
Next, we can compute the vector $z_{\alpha^1}$ as follows: from
\begin{align*}
\beta_{w^k,m}^{\star_1\star_2+} = -\frac{\pi}{8},\quad\beta_{w^k,m}^{\star_1\star_2-} = \frac{\pi}{8},\quad\beta_{w^m,k}^{\star_1\star_2+} = -\frac{\pi}{8}q,\quad\beta_{w^m,k}^{\star_1\star_2-} = \frac{\pi }{8}q,
\end{align*}
with $(\star_1,\star_2)\in \{+,-\}^2$, we get 
\begin{gather}
\begin{split}
 z_{\alpha^1}
&= {\footnotesize\begin{pmatrix}
0 + {L_1}q\left( \beta_{w^m,k}^{+++} - \beta_{w^m,k}^{++-} + \beta_{w^m,k}^{+-+} - \beta_{w^m,k}^{+--} \right) \\
{L_2} \left( \beta_{w^k,m}^{+++} + \beta_{w^k,m}^{++-} \sign(0-1) - \beta_{w^k,m}^{+-+} \sign(0-1)- \beta_{w^k,m}^{+--} \right) + 0 \\
-{L_3} \left( \beta_{w^k,m}^{+++} + \beta_{w^k,m}^{+-+} - \beta_{w^k,m}^{+--} - \beta_{w^k,m}^{++-} \right) - {L_3} \left( \beta_{w^m,k}^{+++} + \beta_{w^m,k}^{+-+} - \beta_{w^m,k}^{+--} - \beta_{w^m,k}^{++-} \right) 
\end{pmatrix}  } \notag\\
&= \frac{\pi }{2}\begin{pmatrix}
- L_1q^2 \\
L_2 \\
L_3(q+1)
\end{pmatrix},\label{sign}
\end{split}
\end{gather}

\begin{remark}
The factors~$\sign(0-1)=\sign(k_2-m_2)$ appearing in~\eqref{sign} are due to the fact that the vector functions~$\YY_z^n$ in~\eqref{cal_YY} are defined for nonnegative frequencies~$n\in\N^3$, and
in~\eqref{CoorYY_Rect} the frequencies may be negative. To guarantee nonnegative frequencies we can just rewrite~\eqref{CoorYY_Rect} by replacing
each $\mathrm S_i(k_i-m_i)$ by its equivalent~$\sign(k_i-m_i)\mathrm S_i(|k_i-m_i|)$. Also, recall that~$\mathrm C_i(|k_i-m_i|)=\mathrm C_i(k_i-m_i)$.
\end{remark}

Next, we choose
\begin{align*}
k &=(1,0,q-1),& m &=(0,1,2),\\ w^k &= (L_1(q-1),0,-L_3),& w^m &= (0,2L_2,-L_3),
\end{align*}
which gives us 
\begin{align*}
\left( Y^k \cdot \nabla   \right) Y^m + \left( Y^m \cdot \nabla   \right) Y^k = \YY^{(1,1,q+1)}_{z_{\gamma^1}} + \YY^{(1,1,q-2)}_{z_{\gamma^2}}, 
\end{align*}
for suitable $z_{\gamma^1}, z_{\gamma^2} \in \R^3$.
Again from assumption~\eqref{IH.R} we have $Y^{\{1,2\},(1,1,q-2)}  \subseteq \GG^{q-1}$, and we can conclude that 
 $\Pi \YY^{(1,1,q+1)}_{z^{\gamma^1}} \in \GG^{q}$. From~\eqref{CoorYY_Rect} we find
\begin{align*}
&~z_{\gamma^1}= \frac{\pi }{2}
\begin{pmatrix}
-  L_1(q-1)^2 \\
-4 L_2 \\
 L_3(q+1)
\end{pmatrix}.
\end{align*}
In order to use Lemma~\ref{lemma-lin-indp}, we observe that the family $\{z_{\alpha^1},~z_{\gamma^1},~(1,1,q+1)\}$  in linearly independent, which follows from
\begin{align*}
\det(n~z_{\alpha^1}~z_{\gamma^1}) &= \frac{\pi^2}{4}
\det \begin{pmatrix}
1 & -L_1 q^2 & -L_1 (q-1)^2\\
1 & -L_2  & -4 L_2 \\
q+1 &L_3(q+1) &  L_3 (q+1)
\end{pmatrix}\\
&= \frac{\pi^2}{4} (q+1) \left( L_1(L_2+L_3)(2q-1) +3L_1L_2q^2 + 3L_2L_3 \right)> 0. 
\end{align*}
Therefore Lemma~\ref{lemma-lin-indp} give us  
\begin{align} \label{res-step21-3Drect}
Y^{\{1,2\},(1,1,q+1)} \subseteq \GG^{q}.
\end{align}

\noindent{\bf Induction step.}
Now let us assume that
\addtocounter{equation}{1}
\begin{equation}\label{IH.R-1l.l1}
Y^{\{1,2 \},(1,l-2, q+1)} \subseteq  \GG^{q}, \quad\mbox{for a given }l,\quad 2\le l\le q.\tag{{\bf IH.R}\ref{Step2_3Drect}-eq.\theequation}
\end{equation}
Notice that~\eqref{one0_R} and~\eqref{res-step21-3Drect} give us
\begin{equation}\label{01_ok_R}
Y^{\{1,2 \},(1,l, q+1)} \subseteq  \GG^{q},\quad\mbox{for all}\quad l\in\{0,1\}.
\end{equation}

In order to generate~$Y^{\{1,2 \},(1,l, q+1)}$ we choose
\begin{align*}
k &=(1,l-1,q),& m &=(0,1,1),\\ w^k &= (0,L_2q,L_3(1-l)),& w^m &= (0,L_2,-L_3),
\end{align*}
This choice gives us 
\begin{align*}
\left( Y^k \cdot \nabla   \right) Y^m + \left( Y^m \cdot \nabla   \right) Y^k = \YY^{(1,l,q+1)}_{z_{\alpha^1}} + \YY^{(1,l-2,q+1)}_{z_{\alpha^2}} + \YY^{(1,l,q-1)}_{z_{\alpha^3}} +  \YY^{(1,l-2,q-1)}_{z_{\alpha^3}}.
\end{align*}
From assumption~\eqref{IH.R} we have that both $Y^{j(1,l,q-1),(1,l,q-1)}$ and \\$Y^{j(1,l-2,q-1),(1,l-2,q-1)}$ belong to $\GG^{q-1}$; and from assumption~\eqref{IH.R-1l.l1}
we have $Y^{j(1,l-2,q+1),(1,l-2,q+1)} \in \GG^{q}$. Thus, we can conclude that $\Pi\YY^{(1,l,q+1)}_{z_{\alpha^1}} \in \GG^{q}$.\\
To compute~$z_{\alpha^1}$ we use
\begin{align*}
\beta_{w^k,m}^{+++} =\beta_{w^k,m}^{-++} =  \frac{\pi }{8}(q-l +1)\quad\text{and}\quad \beta_{w^m,k}^{+++} =\beta_{w^m,k}^{-++} =  \frac{\pi}{8}(l-q-1),
\end{align*}
and obtain
\begin{align*}
z_{\alpha^1}
&= \begin{pmatrix}
0  \\
L_2 \left( \beta_{w^k,m}^{+++} + \beta_{w^k,m}^{-++} \right) + L_2q \left( \beta_{w^m,k}^{+++} + \beta_{w^m,k}^{-++} \right) \\
-L_3. \left( \beta_{w^k,m}^{+++} + \beta_{w^k,m}^{-++} \right) + L_3(1-l). \left( \beta_{w^m,k}^{+++} + \beta_{w^m,k}^{-++} \right)
\end{pmatrix} \\ &= \frac{\pi }{4}
\begin{pmatrix}
0 \\
 L_2(q-l+1)(1-q)\\
 L_3 (q-l+1)(l-2)
\end{pmatrix}.
\end{align*} 
Next, we choose the same frequencies $(k,m)$ with different~$(w^k,w^m)$:
\begin{align*}
k &=(1,l-1,q),& m &=(0,1,1),\\ w^k &= (L_1q,0,-L_3),& w^m &= (0,L_2,-L_3),
\end{align*}
Proceeding as above, we obtain that  $\Pi\YY^{(1,l,q+1)}_{z_{\gamma^1}} \in \GG^{q}$ and,
from
\[\beta_{w^k,m}^{+++} = \beta_{w^k,m}^{-++} = -\frac{\pi}{8}\quad\mbox{and}\quad\beta_{w^m,k}^{+++} = \beta_{w^m,k}^{-++} =  \frac{\pi}{8} (l-q-1),\] we find  
\begin{align*}
z_{\gamma^1}=\frac{\pi }{4}
\begin{pmatrix}
- L_1 q(q-l+1) \\
- L_2 \\
 L_3(q-l+2)
\end{pmatrix}.
\end{align*}
Then, from  
\begin{align*}
\frac{16}{\pi^2}\det(n~z_{\alpha^1}~z_{\gamma^1}) &= \det \begin{pmatrix}
1 & 0 & -L_1q(q-l+1) \\
l & L_2(q-l+1)(1-q)&-L_2  \\
q+1 &L_3(q-l+1)(l-2)&L_3(q-l+2) 
\end{pmatrix} \\
&=  \det \begin{pmatrix}
L_2(q-l+1)(1-q)&-L_2  \\
L_3(q-l+1)(l-2) &L_3(q-l+2) 
\end{pmatrix}  \\
&\quad - L_1q(q-l+1) \det \begin{pmatrix}
l & L_2(q-l+1)(1-q)  \\
q+1 &L_3(q-l+1)(l-2) 
\end{pmatrix} \\ 
& = - L_2L_3q(q-l+1)^2 -  L_1q(q-l+1)^2 \left( L_3l^2 - 2L_3l + L_2q^2 -L_2  \right) \\
&=  -q(q-l+1)^2 \left[ L_2L_3 + L_1L_3l(l - 2)  + L_1L_2(q^2 -1)  \right] < 0, 
\end{align*}
since $2 \le l \le q$, using Lemma \ref{lemma-lin-indp}, we can conclude that $Y^{\{1,2\},(1,l,q+1)} \in \GG^{q}$.

We have just proven that assumption~\eqref{IH.R-1l.l1} leads us to~$Y^{\{1,2\},(1,l,q+1)} \in \GG^{q}$. Then by induction, using~\eqref{01_ok_R}, we can conclude that
\begin{subequations}\label{res-step23-3Drect} 
\begin{align} 
\left\{ Y^{\{1,2 \},(1,l,q+1)}\mid  0< l\leq q  \right\} &\subseteq \GG^{q}, \label{res-step23-3Drect-a} 
\intertext{and by a similar argument we can derive that}
\left\{ Y^{\{1,2 \},(l,1,q+1)}\mid  0< l\leq q  \right\} &\subseteq \GG^{q}. \label{res-step23-3Drect-b}
\end{align}
\end{subequations}

\bigskip\noindent
{$\bullet$~Step~\ref{Step3_3Drect}: {\em~ Generating the family $Y^{j(n),n}$ with $n = (n_1,n_2, q+1)$ where $2 \le n_1 \le q $  and $2 \le n_2 \le q$.}

Again, we proceed by induction on the pair~$(n_1,n_2)$, under the lexicographical order $(n_1,n_2)<(m_1,m_2)$ iff $n_1<m_1$, or~$n_1=m_1$ and~$n_2<m_2$, defined on the set
$N_q\coloneqq \{(\kappa_1,\kappa_2)\in\{0,1,2,\dots,q\}^2\setminus\{0,0\}\}$.

\noindent{\bf Base step.} From~\eqref{one0_R}, \eqref{01_ok_R}, and~\eqref{res-step23-3Drect}, we know that 
\begin{equation}\label{base2_3Drect}
 Y^{j(n),n}\in \GG^{q},\quad\mbox{for all}\quad n=(n_1,n_2, q+1),\quad (n_1,n_2)\in N_q,\quad (0,0)<(n_1,n_2)<(2,2).
\end{equation}

\noindent{\bf Induction step.} Now we assume that 
\addtocounter{equation}{1}
\begin{align}
Y^{j(\kappa),\kappa} \in \GG^{q},\quad\mbox{for all}\quad\kappa\in N_q\quad\mbox{with}\quad
\quad\left\{\begin{array}{l}(0,0)<(\kappa_1,\kappa_2)<(n_1,n_2)\le(q,q),\\(2,2)\le(n_1,n_2),\quad \kappa_3=q+1
                         \end{array}\right.
\label{IH.R-n1.n2}\tag{{\bf IH.R}\ref{Step2_3Drect}-eq.\theequation}
\end{align}
We want to prove that~$Y^{j(n),n} \in \GG^{q}$, with~$n=(n_1,n_2, q+1)$.

By choosing
\begin{align*}
k &=(n_1-1,n_2-1,q),& m &=(1,1,1),\\ w^k &= (0,L_2q,L_3(1-n_2)),& w^m &= (0,L_2,-L_3),
\end{align*}
we find
 \begin{align*}
\left( Y^k \cdot \nabla   \right) Y^m + \left( Y^m \cdot \nabla   \right) Y^k = \YY^{(n_1,n_2,q+1)}_{z_{\alpha^1}} + \sum_{i=2}^8{\YY^{\kappa^i}_{z_{\alpha^i}}},
\end{align*}
with $\{\kappa^i\mid 2\le i\le8\}= \{(n_1-2,n_2-2,q-1), (n_1,n_2-2,q-1), (n_1-2,n_2,q-1),(n_1,n_2,q-1) (n_1-2,n_2-2,q+1), (n_1,n_2-2,q+1), (n_1-2,n_2,q+1) \}$.
From assumption~\eqref{IH.R}, we find that 
$\Pi \YY^{\kappa^i}_{z_{\alpha^i}} \in \GG^{q-1}$,
for $\kappa^i \in \{ (n_1,n_2-2,q-1), (n_1-2,n_2,q-1),(n_1,n_2,q-1) \}$; and assumption~\eqref{IH.R-n1.n2} implies that
$\Pi \YY^{\kappa^i}_{z_{\alpha^i}} \in \GG^{q}$}, for $\kappa^i \in \{(n_1,n_2-2,q+1), (n_1-2,n_2,q+1) \}$.

Now if~$(n_1,n_2)>(2,2)$, then again by assumptions~\eqref{IH.R} \\and~\eqref{IH.R-n1.n2} 
we find that~$\Pi \YY^{\kappa^i}_{z_{\alpha^i}} \in \GG^{q}$, with~$\kappa^i\in\{(n_1-2,n_2-2,q-1),(n_1-2,n_2-2,q-1)\}$. On the other hand if $(n_1,n_2)=(2,2)$, then
$\Pi \YY^{\kappa^i}_{z_{\alpha^i}} =0\in \GG^{q}$, with~$\kappa^i\in\{(n_1-2,n_2-2,q-1),(n_1-2,n_2-2,q+1)\}$.

Thus, we can conclude that $\Pi \YY^{(n_1,n_2,q+1)}_{z_{\alpha^1}} \in \GG^{q}$. Now, from
\[
\beta_{w^k,m}^{+++} = \frac{\pi}{8} (q-n_2+1)\quad\mbox{and}\quad\beta_{w^m,k}^{+++} = \frac{\pi}{8} (n_2-q-1),
\] we obtain
\begin{align*}
z_{\alpha^1} = \frac{\pi}{8}
\begin{pmatrix}
0 \\
L_2(q-n_2+1)(1-q) \\
L_3(q-n_2+1)(2-n_2)
\end{pmatrix}.
\end{align*}
Analogously with the choice
\begin{align*}
k &=(n_1-1,n_2-1,q),& m &=(1,1,1),\\ w^k &= (L_1q,0,L_3(1-n_1)),& w^m &= (0,L_2,-L_3),
\end{align*}
we can conclude that $\Pi \YY^{(n_1,n_2,q+1)}_{z_{\gamma^1}} \in \GG^{q}$ and, from
\[
\beta_{w^k,m}^{+++}  = \frac{\pi}{8} \left( q-n_1 +1 \right)\quad\mbox{and}\quad\beta_{w^m,k}^{+++} =  \frac{\pi}{8} (n_2 - q-1),
\]
we obtain 
\begin{align*}
z_{\gamma^1}=  \frac{\pi}{8}
\begin{pmatrix}
- L_1q(q-n_2+1)\\
 L_2(q-n_1+1) \\
 L_3(1-n_1)(q-n_2+1) + L_3(q-n_1+1)
\end{pmatrix}.
\end{align*}
The family~$\{n,z_{\alpha^1},z_{\gamma^1}\}$ is linearly independent, because
\begin{align*}
&~~ \det \begin{pmatrix}
n_1  & 0 & -L_1q(q-n_2+1)\\
n_2 & L_2(q-n_2+1)(1-q)&L_2(q-n_1+1) \\
q+1 & L_3(q-n_2+1)(n_2-2) &L_3(n_1-1)(q-n_2+1) - L_3(q-n_1+1)
\end{pmatrix} \\
&=  n_1 (q-n_2+1) \det  \begin{pmatrix}
 L_2(1-q) & L_2(q-n_1+1) \\
L_3(n_2-2) & L_3(n_1-1)(q-n_2+1) - L_3(q-n_1+1)
\end{pmatrix}\\
&~~ - L_1q(q-n_2+1)^2 \det \begin{pmatrix}
n_2 &  L_2(1-q) \\
q+1 & L_3(n_2 -2) 
\end{pmatrix}   \\
& = - q(q-n_2+1)^2 \left[ L_2 L_3n_1 (n_1 - 2) + L_1 L_3 n_2 (n_2 - 2)  + L_1L_2(q^2 - 1)\right] <0,
\end{align*}
since $2 \le n_1 \le q$ and $2 \le n_1 \le q$. Thus from Lemma~\ref{lemma-lin-indp} we have that $Y^{\{1,2\},(n_1,n_2,q+1)}\subset\GG^{q}$.

We have just proved that assumption~\eqref{IH.R-n1.n2} implies that
\begin{equation*} 
Y^{\{1,2\},(n_1,n_2,q+1)} \in \GG^{q}.
\end{equation*}
Therefore, using \eqref{base2_3Drect}, by induction it follows that~$Y^{\{1,2\},n} \in \GG^{q}$ with~$n=(n_1,n_2,q+1)$ and~$(n_1,n_2)\in N_q$, which implies that
$Y^{\{1,2\},n} \in \GG^{q}$ for all $n \in \RR^{q+1}_3$. An analogous argument leads us to
\begin{align} \label{res-part1-3Drect}
Y^{\{1,2\},n} \in \GG^{q},  \quad \text{for all~} n \in \RR^{q+1}_1\cup\RR^{q+1}_2\cup\RR^{q+1}_3,
\end{align}
which ends the proof of Lemma~\ref{L:Part1_R}.\null\hfill\null{\qed}

\subsubsection{Proof of Lemma~\ref{L:Part2_R}}.\label{ssS:Part2_R}

We prove that $Y^{j(n),n} \in \GG^{q}$ for $n=(l,q+1,q+1) \in \LL_{2,3}^{q+1}$, $1\le l\le q$.
We choose
\begin{align*}
k &=(l,q-1,q),& m &=(0,2,1),\\ w^k &= (0,L_2q,L_3(1-q)),& w^m &= (0,L_2,-2L_3),
\end{align*}
which leads us to
\begin{align*}
\left( Y^k \cdot \nabla   \right) Y^m + \left( Y^m \cdot \nabla   \right) Y^k = \YY^{(l,q+1,q+1)}_{z_{\alpha^1}}
+  \YY^{(l,q-2, q+1)}_{z_{\alpha^2}} + \YY^{(l,q+1, q-1)}_{z_{\alpha^3}} + \YY^{(l,q-2, q-1)}_{z_{\alpha^4}}.
\end{align*}
By the induction hypothesis~\eqref{IH.R} we have $\Pi\YY^{(l,q-2,q-1)}_{z_{\alpha^4}} \in \GG^{q-1}$. From \eqref{res-part1-3Drect},
since $\{(l,q-2, q+1),\,(l,q+1,q-1)\} \subset \RR^{q+1}_3 \cup \RR^{q+1}_2$, we also have
$\Pi\YY^{(l,q-2, q+1)}_{z_{\alpha^2}}+\Pi\YY^{(l,q+1,q-1)}_{z_{\alpha^3}} \in \GG^{q}$.  Therefore, we obtain that $\Pi \YY^{(l,q+1,q+1)}_{z_{\alpha^1}} \in \GG^{q}$.
Now, from
\[
\beta_{w^k,m}^{+++} = \beta_{w^k,m}^{-++} = \frac{\pi}{8}(q+1)\quad\mbox{and}\quad\beta_{w^m,k}^{+++} = \beta_{w^m,k}^{-++} =  -\frac{\pi}{8} (q+1),  
\]
we obtain
\begin{align*}
z_{\alpha^1}
&= \begin{pmatrix}
0  \\
L_2 \left( \beta_{w^k,m}^{+++} + \beta_{w^k,m}^{-++} \right) + L_2q \left( \beta_{w^m,k}^{+++} + \beta_{w^m,k}^{-++} \right) \\
-2L_3 \left( \beta_{w^k,m}^{+++} + \beta_{w^k,m}^{-++} \right) + L_3(1-q) \left( \beta_{w^m,k}^{+++} + \beta_{w^m,k}^{-++} \right)
\end{pmatrix} \\ &= \frac{\pi}{4}
\begin{pmatrix}
0 \\
L_2(1-q^2) \\
L_3(q+1)(q-2)
\end{pmatrix}.
\end{align*} 
Analogously the choice
\begin{align*}
k &=(l,q-1,q),& m &=(0,2,1),\\ w^k &= (L_1q,0,-L_3l),& w^m &= (0,L_2,-2L_3),
\end{align*}
allow us to conclude that $\Pi \YY^{(l,q+1,q+1)}_{z_{\gamma^1}} \in \GG^{q}$. Where, from
\[
\beta_{w^k,m}^{+++} = \beta_{w^k,m}^{-++} = -\frac{ \pi}{8}l\quad\mbox{and}\quad\beta_{w^m,k}^{+++} = \beta_{w^m,k}^{-++} =  -\frac{\pi}{8} (q+1), 
\]
we have 
\begin{align*}
&~z_{\gamma^1}
= \frac{\pi}{4}
\begin{pmatrix}
-L_1 q(q+1) \\
-L_2 l \\
L_3 l(q+3)
\end{pmatrix}.
\end{align*} 
Now from Lemma~\ref{lemma-lin-indp} and 
\begin{align*}
\det(n~z_{\alpha^1}~z_{\gamma^1})&= \det \begin{pmatrix}
l & 0 & -L_1q(q+1)\\
q+1 &L_2(1-q^2) &-L_2l  \\
q+1 &L_3(q+1)(q-2) &L_3l(q+3)
\end{pmatrix} \\  &= -q (q+1)^2 \left[ L_2L_3l^2 + L_1L_3(q+1)(q-2) + L_1L_2(q^2-1) \right] < 0, 
\end{align*}
because $l \ge 1$ and $q \ge 3$, it follows that~$Y^{\{1,2 \},(l,q+1,q+1)} \in \GG^{q}$, for ~$1 \le l \le q$. A similar argument gives us

\begin{align} \label{res-part2-3Drect}
Y^{\{1,2\},n} \in \GG^{q},  \quad \text{for all}\quad n \in \LL^{q+1}_{1,2} \cup \LL^{q+1}_{2,3} \cup \LL^{q+1}_{3,1},
\end{align}
which ends the proof of Lemma~\ref{L:Part2_R}.\null\hfill\null{\qed}

\subsubsection{Proof of Lemma~\ref{L:Part3_R}.}\label{ssS:Part3_R} 

Firstly, we choose
\begin{align*}
k &=(q,q-1,q),& m &=(1,2,1),\\ w^k &= (0,L_2q,L_3(1-q)),& w^m &= (0,L_2,-2L_3),
\end{align*}
which give us 
\begin{align*}
\left( Y^k \cdot \nabla   \right) Y^m + \left( Y^m \cdot \nabla   \right) Y^k = \YY^{(q+1,q+1,q+1)}_{z_{\alpha^1}} + \sum_{i=2}^8{\YY^{\kappa^i}_{z_{\alpha^i}}},
\end{align*}
where $\{\kappa^i\mid i\in\{2,\cdots,8\}\}=  \{ (q+1, q+1, q-1),\, (q-1, q+1, q+1)),\,(q-1, q-2, q+1),~(q-1, q+1, q-1),~(q+1, q-2, q-1),~(q-1, q-2, q+1),\,(q-1, q-2, q-1) \}$.\\
Since
\[
\{\kappa^i\mid i\in\{2,\cdots,8\}\}\subseteq (\RR^{q+1}_1\cup\RR^{q+1}_2\cup\RR^{q+1}_3)\bigcup(\LL^{q+1}_{1,2} \cup \LL^{q+1}_{2,3})\bigcup\sS^q 
\]
from~\eqref{IH.R}, \eqref{res-part1-3Drect}, and~\eqref{res-part2-3Drect} we can conclude that~$\Pi \YY^{(q+1,q+1,q+1)}_{z_{\alpha^1}} \in \GG^{q}$.
 
Now, from the identities
\[
 \beta_{w^k,m}^{+++} = \frac{\pi}{8}(q+1)\quad\mbox{and}\quad\beta_{w^m,k}^{+++} =  -\frac{\pi}{8} (q+1),
\]
we obtain 
\begin{align*}
z_{\alpha^1} = \begin{pmatrix}
0  \\
L_2.\beta_{w^k,m}^{+++}  + L_2q.\beta_{w^m,k}^{+++}  \\
-2L_3.\beta_{w^k,m}^{+++} + L_3(1-q).\beta_{w^m,k}^{+++} 
\end{pmatrix} = \frac{\pi}{8}
\begin{pmatrix}
0 \\
 L_2(1-q^2) \\
 L_3(q+1)(q-2)
\end{pmatrix}.
\end{align*} 

Next by choosing
\begin{align*}
k &=(q,q-1,q),& m &=(1,2,1),\\ w^k &= (L_1(1-q),L_2q,0),& w^m &= (0,L_2,-2L_3),
\end{align*}
and proceeding as above, we can conclude that $\Pi \YY^{(q+1,q+1,q+1)}_{z_{\gamma^1}} \in \GG^{q}$, with
\begin{align*}
z_{\gamma^1}  = \begin{pmatrix}
0\beta_{w^k,m}^{+++} + L_1(1-q)\beta_{w^m,k}^{+++}  \\
L_2 \beta_{w^k,m}^{+++}  + L_2q \beta_{w^m,k}^{+++}  \\
-2L_3.\beta_{w^k,m}^{+++} + 0  \beta_{w^m,k}^{+++}
\end{pmatrix} =  \frac{\pi}{8}
\begin{pmatrix}
 L_1(q^2-1)  \\
 L_2 (1-q^2) \\
- 2L_3(q+1)
\end{pmatrix}.
\end{align*} 

With~$n=(q+1,q+1,q+1)$, using again Lemma \ref{lemma-lin-indp} and
\begin{align*}
&\det(n~z_{\alpha^1}~z_{\gamma^1})=\frac{\pi^2}{64}(q+1)^3\det \begin{pmatrix}
1 & 0 & L_1(q-1) \\
1 &L_2(1-q) &L_2(1-q)   \\
1 &L_3(q-2) &-2L_3
\end{pmatrix} \\ &\qquad= \frac{\pi^2}{64}(q+1)^3 \left[ 
L_2L_3 \det  \begin{pmatrix}
1-q & 1-q \\
q-2 & -2
\end{pmatrix}  + L_1(q-1) \det  \begin{pmatrix}
1 & L_2(1-q) \\
1 & L_3(q-2)
\end{pmatrix} 
\right]\\
&\qquad=\frac{\pi^2}{64}(q+1)^3(q-1) \left[ (L_1L_2+L_2L_3) (q-1) + L_1L_3 (q-2) \right]> 0,
\end{align*}
because $q\ge 3$, we obtain 
\begin{align} \label{res-part3-3Drect}
Y^{\{1,2\},(q+1,q+1,q+1)} \subset \GG^{q},
\end{align}
which ends the proof of Lemma~\ref{L:Part3_R}.\null\hfill\null{\qed}

\section{Final Remarks} 
\label{S:FinRe}
We proved the approximate controllability of the Navier--Stokes system in a $\mathrm{3D}$ rectangle by degenerate (low modes) forcing,
under Lions boundary conditions. We used the analogous $\mathrm{2D}$ result, derived in~\cite{Rod-Thesis08} (see also~\cite{Rod06} for $({\tt B},\D(A))$-saturating sets).
In~\cite{PhanRod-ecc15}
the case of a~$\mathrm{2D}$ cylinder is considered, thus we may wonder whether we can also derive the approximate controllability for the case of a~$\mathrm{3D}$ cylinder.
This case can be seen as the case where the fluid is contained in a long (infinite) $\mathrm{3D}$ channel with Lions boundary conditions, and with the periodicity
assumption on the long (infinite)
direction, thus it is a case of interest for applications. First computations show that the existence of a $({\tt L},\D(A))$-saturating set in this case is plausible, but
the computations details are still to be checked. Since those computations will be long, and since this manuscript is already long, we will investigate the case of a $\mathrm{3D}$ cylinder in a future work.

\smallskip
We underline that the presented saturating set is (by definition) independent of the viscosity coefficient~$\nu$. That is, approximate controllability holds by means of controls taking values
in~$\GG^1=\linspan(\CC)+\linspan{\BB(\CC,\linspan\CC)}=\linspan\left(\CC\bigcup\BB(\CC,\CC)\right)$, for any~$\nu>0$.
It is plausible that a $({\tt L},\D(A))$-saturating set with less elements does exist, but
it is not our goal here to minimize the number of elements of~$\CC$.

\smallskip
We have used the result in~\cite{Shirikyan06} where it is proven that under Dirichlet boundary conditions
the existence of a $({\tt B},\D(A))$-saturating set implies
the approximate controllability of Navier--Stokes system by degenerate forcing. We can conclude from our results that the same controllability result follows from the
existence of a~$({\tt L},\D(A))$-saturating set. However, up to our knowledge, neither the existence of a $({\tt B},\D(A))$-saturating set nor that of a
$({\tt L},\D(A))$-saturating set is known under Dirichlet boundary conditions. That is, essentially the approximate controllability of the Navier--Stokes system
is still an open problem under Dirichlet boundary conditions. Therefore, it is of interest to find a
saturating set for such classical boundary conditions, because they are the most realistic in many situations.

Up to now the known examples of saturating sets consist of eigenfuntions of the Stokes operator. For applications, it would be interesting to consider more realistic
functions as actuators, as locally supported functions, recall~\cite[Problem~VII]{Agrachev13_arx} (cf.~\cite[Section~V]{PhanRod-ecc15}). Furthermore,
the explicit expressions for
the Stokes operator may be not available as it is the case (up to our best knowledge) for Dirichlet boundary conditions. Here, and in previous works the density of
$\mathop{\bigcup}_{j\in\N}\GG^j$ in~$H$ has been proven by showing the the union contains all the eigenfunctions of the Stokes operator. Thus this argument may be difficult
(maybe, not possible) to use in the case of Dirichlet boundary conditions, which makes the investigation of these last boundary conditions an interesting problem.

\medskip\noindent
{\bf Acknowledgments.}
The authors acknowledge partial support from the Austrian Science Fund (FWF): P 26034-N25. D. Phan appreciates partial support from the Foundation of Tampere University of Technology.

\end{document}